\documentclass[11pt]{article}

\usepackage{amsmath,amsthm,verbatim,amssymb,amsfonts,amscd, graphicx}
\usepackage{graphics,tikz-cd}
\usepackage{enumitem,kantlipsum}
\theoremstyle{plain}
\theoremstyle{plain}
\newtheorem{thm}{Theorem}[section] % reset theorem numbering for each chapter
\newtheorem{lemma}[thm]{Lemma}
\newtheorem{prop}[thm]{Prop}

\newtheorem{rmk}[thm]{Remark}
\theoremstyle{definition}
\newtheorem{defn}[thm]{Definition} % definition numbers are dependent on theorem numbers

\newcommand{\Q}{\mathbb{Q}}
\newcommand{\OO}{\mathcal{O}}
\newcommand{\R}{\mathbb{R}}
\newcommand{\N}{\mathbb{N}}
\newcommand{\simR}{\sim_{\R}}
\newcommand{\simQ}{\sim_{\Q}}

\newcommand{\proofsketch}{\vspace*{-1ex} \noindent {\bf Proof Sketch: }}

\newcommand{\ceiling}[1]{{\left\lceil{#1}\right\rceil}}
\newcommand{\floor}[1]{{\left\lfloor{#1}\right\rfloor}}

\begin{document}
 
\clearpage 
\pagenumbering{gobble}
\title{Summary on Proof of BAB}
\author{Yanning Xu}
\maketitle
\tableofcontents
\newpage
\pagenumbering{arabic}
\clearpage
\setcounter{page}{1}
\section{Introduction}

This write-up contains most of the technical proofs (some sketch but mostly in details) for the two BAB papers. For full details of all the proof, please see \cite{fa} and \cite{BAB}. Some proofs are reordered and regrouped in order to help to understand of the proof of main results and also reduce the length of the proof. Also, some motivation, such as why the lemma is formulated as it is, is given (mostly in step 0 of the proof). \\I would like to thank Professor Caucher Birkar for his help for answering countless questions from me and explaining the motivation behind things. \\
All varieties here are assumed to be quasi-projective over a fixed algebraically closed field of characteristic zero. Also $\N$ here denotes the set of positive integers. Here we first state all the main theorems in the two BAB papers. The definition will be given in the next section.
\begin{thm}[BAB]\label{BAB}
    Let $d\in \N$ and $\epsilon>0$. Then the set of projective varieties $X$ such that $(X,B)$ is $\epsilon$-lc of dimension d for some boundary $B$ and $-(K_X+B)$ is nef and big, form a bounded family.
\end{thm}
\begin{thm}[Effective Birationality]\label{eff_bir}
    Let $d\in\N$ and $\epsilon>0$, then there exists $m(d,\epsilon)\in \N$ such that if $X$ is any $\epsilon$-lc weak fano variety of dimension $d$, then $|-mK_X|$ defines a birational map.
\end{thm}
\begin{thm}\label{wea_bab}
    Let $d\in \N$ and $\epsilon,\delta>0$. Then the set of $X$ such that there is a boundary $B$ with $(X,B)$ $\epsilon$-lc of dimension d, $B$ big, $K_X+B\simR 0$ and $|B|\geq \delta$, forms a bounded family.
\end{thm}

\begin{thm}\label{com_gen}
    Let $d,p\in \N$ and $\mathfrak{R}\subset [0,1]$ be a finite set of rationals, then there exist $n(d,p,\mathfrak{R})\in \N$ such that if $(X',B'+M')$ projective generalised lc of dimension d, $B'\in \Phi(\mathfrak{R})$, $pM$ b-Cartier, $X'$ fano type and $-(K_{X'}+B'+M')$ is nef, then $K_{X'}+B'+M'$ has a n-complement $K_{X'}+B'^++M'$ with $B'^+\geq B'$.
\end{thm}
\begin{thm}\label{com_rel}
    Let $d\in \N$ and $\mathfrak{R}\subset [0,1]$ be a finite set of rationals, then there exist $n(d,\mathfrak{R})\in \N$ such that if $(X,B)$ a lc pair of dimension $d$, with a contraction $X\rightarrow Z$ with $\dim Z>0$, $B'\in \Phi(\mathfrak{R})$, $X$ Fano type over $Z$ and $-(K_X+B)$ is nef over $Z$. Then for any $z\in Z$, there is a n-complement $K_X+B^+$ for $K_X+B$ over z with $B^+\geq B$
\end{thm}
\begin{thm}\label{bou_exc}
    Let $d\in \N$ and $\mathfrak{R}\subset [0,1]$ be a finite set of rationals. Consider $(X',B'+M')$ as in \ref{com_gen} that are exceptional, then all such $(X'+B')$ form a log bounded family.
\end{thm}
\begin{thm}\label{1.4}
    Let $d,r \in \N$ and $\epsilon>0$ a real number. Then there exists $t(d,\epsilon)>0$ such that if $(X,B)$ is a projective $\epsilon$-lc pair of dimension $d$ and $A:= -(K_X+B)$ is nef and big, then $lct(X,B,|A|_{\R})>t$ 
\end{thm}

\begin{thm}\label{1.6}
    Let $d,r \in \N$ and $\epsilon>0$ a real number. Then there exists $t(d,r,\epsilon)>0$ such that if $(X,B)$ projective $\epsilon$-lc of dimension $d$ and $A$ very ample with $A-B$ ample and $A^d\leq r$, then $lct(X,B,|A|_{\R})>t$.
\end{thm}
\newpage

\section{Preliminary}
\subsection{Some Definitions}
Here we list a few basic definitions for the minimal model program. For more detailed treatment please see \cite{fa} and \cite{km}. We assume the reader is familiar with standard definition involved in the minimal model program.
\begin{enumerate}[wide, labelwidth=!, labelindent=10pt]
\item A contraction is a projective morphism $f:X\rightarrow Y$ such that $f_* \mathcal{O}_X = \OO_Y$, in particular $f$ has connected fibres.
\item Hyperstandard sets: For $\mathfrak{R}\subset[0,1]$, we define $\Phi(\mathfrak{R}) := \{1-\frac{r}{n} | r\in \mathfrak{R},n\in \N\}$. If $\mathfrak{R}$ is a finite set of rationals then $\Phi(\mathfrak{R})$ is a DCC set with only accumulation point 1. 
\item Divisors: We use the usual definition of (Cartier) divisor, $\Q$ (Cartier) divisors, $\R$ (Cartier)-divisors. For a divisor $M=\sum_i a_i M_i$, $M_i$ prime divisors, we define $M^{\geq a} := \sum_{i, a_i\geq a} a_i M_i$. Also we write $M\geq a$ if $a_i\geq a$ $\forall i$. Similar definition can be made for $\leq,<,>$. Also we define $\sim, \simQ,\simR$ and their relative version in the usual sense. Also for $\R$-divisor M, we define linear system $|M| := \{N\geq 0 | N\sim M\}$ and $H^0(M):= \{f\in K(X)|(f)+M\geq 0\}$. Also define the $\R$ linear system of $M$ to be $|M|_{\R} := \{N\geq 0| N\simR M\}$. Given a morphism $Y\rightarrow X$, we usually denote $A_Y$ to mean the pull-back of $A$ to $Y$ where $A$ is a divisor on $X$. If $B$ is a divisor on $Y$, we usually denote $B_X$ to be its push forward to $X$. We have the following lemma due to Hironaka resolution of singularities which is used in the proof of Effective Birationality very often.
\begin{lemma}\label{2.6}(\cite{fa} 2.6) 
    X normal variety and $M$ a $\R$-Cartier $\R$ divisor. Let $f: Y\rightarrow X$ a projective birational morphism with $Y$ normal and let $M_Y := f^* M$. and let $F$ be fixed part of $|M|$ and $F_Y $ be fixed part of $|M_Y|$. Then clearly we have $f_* |M_Y|=|M|$ and $f_*F_Y =F$. Furthermore, we have if $f$ is a sufficiently high resolution, $|M_Y-F_Y|$ is base point free.
\end{lemma}
\item $(X,B)$ is called a sub pair if X is normal variety, $B$ a $\R$ divisor such that $K_X+B$ is $\R$-Cartier. $B$ is called a boundary if $B\leq 1$. $(X,B)$ is called a pair if $B\geq 0$. We define the notion of lc klt plt dlt same as most places. (\cite{fa},2.8). We say $D$, a divisors over $X$, is a non-klt place if $a(D,X,B)\geq 0$ and centre of $D$ on $X$ is called a non-klt centre.
\item Let $(X,B)$ be a pair with a contraction $X\rightarrow Z$, we say $(X,B)$ is log fano over $Z$ if $-(K_X+B)$ ample over $Z$, we say it is weak log fano over $Z$ if $-(K_X+B)$ is nef and big over $Z$. If $Z$ is a point we omit over $Z$. When we don't mention $B$, we assume $B=0$. We also define $X$ fano type over $Z$ if $(X, B)$ is klt weak log fano over $Z$ for some choice of $B$. This is clearly equivalent to there exist $\Gamma$ big $/Z$ such that $(X,\Gamma)$ klt and $K_X+\Gamma \simR 0/Z$. \\ Furthermore, if $X$ is fano type over $Z$, we can run an MMP on any divisors $D$ and the MMP will terminate because it is the same as running MMP on $K_X+(\Gamma+\frac{1}{N}D)$ and $(\Gamma+\frac{1}{N}D)$ is big and the pair has klt singularities for sufficiently large $N$. Also note that outcome of MMP is also a fano type variety over $Z$ due to the following fact: (\cite{fa},2.12) If $X$ fano type and $X\rightarrow Y$ a contraction with $\dim Y>0$, then $Y$ is also fano type.
\item When $X$ is $\Q$-factorial fano type variety, we can firstly run a MMP on $-K_{X}$ which ends with $X'$ a weak fano variety. Now since abundance hold for fano type,  $-K_{X'}$ is semi-ample, hence defines a birational contraction $X'\rightarrow X''$. Now $X''$ is a fano variety since $-K_{X''}$ is ample because it is the pullback of ample divisor along a finite map (we construct the contraction defined by $-K_{X'}$ using stein factorization theorem).
\end{enumerate} 
%%%%%%%%%%%%%%%%%%%%%%%%%%%%%%%%%%%%%
%%%%%%%%%%%%%%%%%%%%%%%%%%%%%%%%%%%%%
\subsection{Generalised Pairs}
Here we define Generalised pairs and discuss some of its properties. 
\begin{defn}
    A generalised pair is given as $(X',B'+M')$ where $X'$  is a normal variety with a projective morphism $X'\rightarrow Z$, $B'\geq 0$ an $\R$ divisor on $X'$ and b-$\R$-Cartier b-divisor represented by some projective birational morphism $X \rightarrow X'$ and $\R$-Cartier $M$ on X such that $M$ is nef/$Z$ and $M|_{X'}=M'$, and $K_{X'}+B'+M'$ is $\R$-Cartier. When $Z$ is a point we say the pair is projective. Note since $M$ is defined birationally, we can always replace $X$ by a log resolution and replace $M$ by its pullback.
\end{defn}
Now we can define singularities by: Replace $X\xrightarrow{\phi}X'$ by a log resolution of $(X',B')$, then we can write $K_X+B+M=\phi^*(K_{X'}+B'+M')$. For a prime divisor $D$ on $X$, we define $a(D,X'B'+M') := 1-\mu_D(B)$. We say $(X',B'+M')$ is generalised klt (lc, $\epsilon$-lc)if $a(D,X'B'+M')>0$ ($\geq 0$, $\geq \epsilon$). We say $(X',B'+M')$ generalised dlt if it is generalised lc and $(X',B')$ is dlt and all non-klt centre of $(X',B'+M')$ is a non-klt centre of $(X',B')$, when $\floor{B'}$ is irreducible, we say it is generalised plt. We can define generalised lc threshold similarly. We also have a similar statement for connectedness principle when $-(K_{X'}+B'+M')$ is nef and big over $Z$. Also when $(X',B'+M')$ is generalised klt and $-(K_{X'}+B'+M')$ is nef and big, then $X'$ is fano type.
\subsection{$\mathbb{Q}$-Factorial Dlt Model, Extraction of Divisors and Plt Blowup of Pairs and Generalised Pairs}\label{Q_fac}
\begin{enumerate}[wide, labelwidth=!, labelindent=10pt]
    \item We construct $\Q$-factorial dlt model for a lc-pair $(X,B)$. Similar argument works for generalised pair. Let $(X,B)$ be a lc-pair. Let $W\xrightarrow{\phi} X$ be a log resolution of $(X,B)$, let $\Sigma := B^\sim +\sum_i E_i$ where $E_i$ are all exceptional prime divisors. Note by lc, we have $K_W+\Sigma = \phi^*(K_X+B)+G$, where $G:= \sum_i a(E_i,X,B)E_i\geq 0$. Now run a MMP on $K_W+\Sigma$ over $X$, we reach a model $Y\xrightarrow{\psi} X$ such that $K_Y+\Sigma_Y$ is a limit of movable divisors over $X$ (i.e. after we have done all divisorial contractions). We claim $G_Y=0$:  Indeed we have $G_Y \simR K_Y+\Sigma_Y /X$, which implies $G_Y\leq 0$ by generalised negativity. This means that $K_Y+\Sigma_Y = \psi^*(K_X+B)$ and in particular $Y$ is $\Q$-factorial dlt by construction and all exceptional divisors of $\psi$ appear as coefficient 1 in $\Sigma_Y$, we call $(Y,\Sigma_Y)$ is a $\Q$-factorial dlt model for $(X,B)$. 
    \item Small $\Q$-factorization and plt blow up. Note if $(X,B)$ is klt, running above argument gives a small klt $\Q$-factorisation of $(X,B)$. If in addition $X$ is $\Q$-factorial, then the above process will simply give $X$, since there are no non-isomorphic small contraction between $\Q$-factorial varieties). 
    \item Extraction of divisors: If $(X,B)$ is klt and we have $a(D,X,B)=\epsilon>0$, then we can by subtracting $a(D,X,B)D$ from $\Gamma$ above we can make sure $\Gamma$ is not contracted and hence get a birational morphism $Y\rightarrow X$ that only extract $D$, (i.e. the only exceptional divisor is $D$). We refer to this as a extraction of divisor $D$. 
    \item plt blowup: Assume $(X,B)$ lc but not klt and $(X,0)$ is $\Q$-factorial klt. First we can obtain $(Y,B_Y)$ a $\Q$-factorial dlt model of $(X,B)$ with $K_Y+B_Y = \phi^*(K_X+B)$, where $\phi: Y\rightarrow X$. Now let $\Gamma_Y $ to be exceptional divisor of $\phi$. Running an MMP on $K_Y+\Gamma_Y$ over $X$ end with $X$ by (2). By construction we have $\Gamma_Y \leq B_Y$. Now replacing $(Y,B_Y)$ with the domain of the last divisorial contraction of the MMP gives a birational morphism $Y\rightarrow X$ contracting only one divisor T, a component of $\floor{B_Y}$ where $K_Y+B_Y$ is pullback of $K_X+B$. Furthermore, $(Y,T)$ is plt and $-(K_Y+T)$ is ample over $X$ by construction, we call such $Y$ a plt blowup of $(X,B)$.
\end{enumerate}

\subsection{Bounded Families of Pairs and Subvarieties}\label{bnd_fam}
For the definition of a bounded family of varieties or log bounded family of couples, please see \cite{fa}. Here we introduce some important property that we will need. The general principle is that any numerical data of bounded family are bounded and most birational modification of bounded family are still bounded. We use the standard notation that we assume $V_i\rightarrow T_i$ gives a bounded family given by the fibers. 
    \begin{enumerate}[wide, labelwidth=!, labelindent=10pt]
    \item Criteria for boundedness: Let $\mathcal{P}$ be a set of couples of dimension d. $\mathcal{P}$ is log bounded if and only if there exist $r\in N$ such that for any $(X,D)\in \mathcal{P}$, $\exists A$ very ample on $X$ such that $A^d\leq r$ and $A^{d-1}D\leq r$. 
    \item Boundedness of log resolution and $\Q$-factorial dlt models: Let $\mathcal{P}$ be a bounded family of couples of dimension $d$, then we can choose log resolution $\phi: W\rightarrow X$ for $(X,D)\in \mathcal{P}$ such that the set of $(W,D_W)$ form a bounded family, where $D_W := D^\sim+exc/X$. The same hold for $\Q$-factorial dlt models. (Consider taking log resolution of the family and apply induction on dimension)
    \item Boundedness of numerical invariant of $K_X$: If $\mathcal{P}$ is a bounded family, then most numerical property of $K_X$ is bounded, e.g. Take a very ample divisor, $A$, as in (1), then $A^{d-1}K_X$ is bounded from above and below. (Consider very ample divisor on $V$ and $K_V$)
    \item Boundedness of Cartier index :Let $\mathcal{P}$ be a bounded set of couples of dimension $d$. Then there exist $I(\mathcal{P})\in \N$ such that if X has klt singularities and $M\geq 0$ a integral $\Q$-Cartier divisor such that $(X,Supp M)\in \mathcal{P}$, then $IK_X,IM$ are both Cartier. \\Now also let $r\in \N$, then there exist $J(d,r)\in \N$ such that: Assume further to above, we have $L$ nef integral divisor such $A^{d-1}L\leq r$ and $A$ is very ample $A^d\leq r$ (exist since $X$ is bounded), then $JL$ is Cartier.
    \end{enumerate}
\subsection{Exceptional and Non-Exceptional Pairs}\label{7}
Let $(X',B'+M')$ be a projective generalized pair such that $K_{X'}+B'+M'+P'\simR 0$ for some $P'\geq 0$. We say it is non-exceptional (strongly non-exceptional) if there exist such $P'$ such that $(X',B'+P'+M')$ is not generalised klt (lc). We have an obvious lemma that helps to keep track of exceptionality.
\begin{lemma}
    Let $X\rightarrow X',X''$ be 2 projective morphism $M$ a divisor on $X$. Assume we have $(X',B'+M')$ and $(X'',B''+M'')$ 2 generalised pair, and we have $(K_{X'}+B'+M')|_{X}\leq (K_{X''}+B''+M'')|_{X} $. If $(X',B'+M')$ is exceptional, then so is $(X'',B''+M'')$.
\end{lemma}
\subsection{Complements}
Here we define complements for  pairs. Let $(X,B)$ be a pair where B is a boundary and let $X\rightarrow Z$ be a contraction, let $T := \floor{B},\Delta := B-T$. A n-complement of $K_X+B$ over $z\in Z$ is of the form $K_X+B^+$, such that over some neighbourhood of $z$, we have $(X,B^+)$ is lc, $n(K_X+B^+)\sim 0$ and $nB^+\geq nT+\floor{(n+1)\Delta}$. Clearly if all condition are satisfied except the last condition but we have $B^+\geq B$ then $K_X+B^+$ is also a n-complement. In this case we have $n(B^+-B)$ is an element in $|-n(K_X+B)|$ over $z$. We can also extend this definition to generalised pairs. Instead of requiring $n(K_X+B^+)\sim 0$, we require $n(K_{X'}+B'^+ +M')\sim 0$ and $nM$ is b-Cartier. We have the following useful lemma to keep track of complements.
\begin{lemma}\label{6.1.3}
    1. Let $(X',B'+M')$ with data $X\xrightarrow{\phi}X$ and M on $X$, and $(X'',B''+M'')$ be 2 generalised pair. Assume (by replacing X), $X\xrightarrow{\phi} X',X\xrightarrow{\psi}X''$ be a common resolution such that $\psi_* M=M''$.Suppose further that $\phi^*(K_{X'}+B'+M') + P = \psi^*(K_{X''}+B''+M'')$ for some $P\geq 0$, then if $(X'',B''+M'')$ has a n-complement, then so does $(X',B'+M')$.\\    
    2. Let $(X',B'+M')$ be a generalised pair. Assume $X'\dashrightarrow X''$ be a partial MMP on $-(K_{X'}+B'+M')$, and let $B'',M''$ be pushdown of $B',M'$ to $X''$. If $(X'',B''+M'')$ has a n-complement $B''^+$, then $(X'B'+M')$ has a n-complement.
\end{lemma}
\begin{proof}
    Clearly 1 implies 2 by basic property of MMP. Let $B''^+\geq B''$ be a n-complement for $(X'',B''+M'')$, Consider $B'^+ := B'+\phi_*(P+\psi^*(B''^+-B'')$, then by easy computation, we get $n\phi^*(K_{X'}+B'^+ + M')= n\psi^*(K_{X''}+B''^+ +M'')\sim 0$. Hence we get $n(K_{X'}+B'^+ + M')\sim 0$ and $(K_{X'}+B'^++M')$ is also generalised lc since $(X'',B''^+ +M'' )$ is generalised lc.
\end{proof}
\subsection{Potentially Birational Divisors}\label{2.30}
Let $X$ be a normal projective variety and $D$ a big $\Q$-Cartier $\Q$-divisor on $X$, we say $D$ is potentially birational if for any $x,y$ general closed point in $X$, there is a $0\leq \Delta\simQ (1-\epsilon)D$ for some $\epsilon\in(0,1)$ such that $(X,\Delta)$ is lc at $x$ with non-klt centre $\{x\}$  and not klt at $y$. If $D$ is potentially birational, then $|K_X+\ceiling{D}|$ defines a birational map.
\subsection{Volume, Numerical Kodiara Dimension of Divisors}\label{2.38}
For $X$ be a normal projective variety of dimension $d$ and $D$ a $\R$ divisor on X. we define $vol(D) := \limsup_{m\rightarrow \infty} \frac{h^0(\floor{mD})}{m^d/d!}$, we note that $vol(D)>0$ if and only if $D$ is big. For $D$ a $\R$-Cartier $\R$ divisor we define the numerical Kodiara dimension, $\kappa_\sigma(D) $, to be $-\infty$ if $D$ not pseudoeffective, and to be the largest integer $r$ such that $\limsup_{m\rightarrow \infty}\frac{h^0(\floor{mD}+A)}{m^r}>0$. We have the following two lemmas which will be used in proof of effective birationality. The following two lemmas essentially follows from the fact that volume and $\kappa_\sigma$ is essentially invariant for families of varieties defined by fibres of some morphism.
\begin{lemma}
    Let $\mathcal{P}$ be a bounded set of smooth projective varieties $X$ with $\kappa_\sigma(K_X)=0$, then there exist $l\in \N$ such that $h^0(lK_X)\neq 0$ for all $X\in \mathcal{P}$.
\end{lemma}
\begin{lemma}
    Let $X$ be a $\Q$-factorial projective variety with $D$ be such that $\kappa_\sigma(D)>0$, then for any $A$ ample $\Q$-divisor we have $\lim_{n\rightarrow \infty}vol(mD+A)=\infty$. Now further assume $X$ is dimension d and $A$ is very ample and $(X,A)$ belong to a bounded family $\mathcal{P}$, then for all $q>0$, there is a $p(\mathcal{P})\in \N$ such that $vol(pK_X+A)>q$.
\end{lemma}

\subsection{Pairs with Large Boundary}
The following lemma is extremely important to bounding volumes which are used in both papers.
\begin{lemma}\label{2.46}
    Let $(X,B)$ be a $\Q$-factorial dlt pair of dimension $d$ and let $M$ a nef Cartier divisor. Let $a>2d$. Then any MMP on $K_X+B+aM$ is $M$-trivial. If M is nef and big Cartier divisor. Then $K_X+B+aM$ is big.
\end{lemma}
\begin{proof}
    The MMP is $M$ trivial followings clearly from $M$ is nef and Cartier, and boundedness of extremal rays. (Hence we can run MMP on $K_X+B+aM$). It suffices to show $K_X+aM$ is big.  If it is not, then there exists $2d<a'<a$ such that $K_X+a'M$ is not pseudo effective. Now running MMP on $K_X+a'M$ end with a mori fiber space $Y\rightarrow T$. Again by boundedness of extremal rays, we get $M_Y\equiv 0/T$, but this contradicts the bigness of $M_Y$.
\end{proof}
\subsection{Consequence of ACC of Log Canonical Threshold} \label{2.50}
Now we will state and prove 2 lemmas which will be crucial for our construction of complements in \ref{6.15}. 
\begin{lemma}
    Let $d\in \N$ and $\Phi\subset[0,1]$ be a DCC set. Then there is $\epsilon>0$ depending only on $d,\Phi$ such that if $(X,B)$ is a lc-projective pair of dimension d, $(X,0)$ klt,  $K_X+B\simR0 $ and $B\in \Phi$ and $D$ a divisor over $X$ with $a(D,X,B)<\epsilon$, then $a(D,X,B)=0$.
\end{lemma}
\begin{proof}
    Assume the theorem is false and we have $\epsilon_i:= a(D_i,X,B) \rightarrow 0$ and $(X_i,B_i)$ as in the claim. Since $(X_i,0)$ klt and we can apply \ref{Q_fac}, we get an birational morphism (maybe ideneity map) $X_i'\rightarrow X_i$ extracting only $D$. Hence let $K_{X_i'}+B_i'$ be the pull back of $K_{X_i}+B_i$ where $B_i' := B_i^\sim +(1-\epsilon_i)D_i$. Hence replacing $X_i$ by $X_i'$ and $\Phi$ accordingly we can assume $D_i$ are divisors on $X$. This contracts ACC of log canonical threshold since $lct(D_i,X_i,B_i) =\epsilon_i\rightarrow0$.
\end{proof}
The next lemma is extremely important in the sense that it allows to replace a DCC set with a finite set when dealing with complements. We will only state and proof the absolute version, for relative version, see \cite{fa}.
\begin{lemma}
    Let $d,p\in \N$ and $\Phi\subset[0,1]$ be a DCC set. then there is $\epsilon(d,p,\Phi)>0$ such that if $(X',B'+M')$ is a generalised pair of dimension d, $B'\in \Phi\cup[1-\epsilon,1]$, pM is b-Cartier, -$(K_{X'}+B'+M')$ is a limit of movable divisor and there is $0\leq P'\simR -(K_{X'}+B'+M')$ such that $(X',B'+P'+M')$ is generalised lc and $X'$ is fano type then if we let $\theta' := B'^{\leq 1-\epsilon}+\floor{B'^{>1-\epsilon}} $,  then run MMP on $-(K_{X'}+\theta'+M')$ and we end with $X''$, then $(X',\theta'+M')$ is generalised lc, the MMP doesn't contract and component of $\floor{\theta'}$,$ -(K_{X''}+\theta''+M'')$ is nef and $(X'',\theta''+M'')$ is generalised lc. In particular, if $K_{X''}+\theta''+M''$ has a n-complement, then so does $K_{X'}+B'+M'$ and coefficient of $\theta''$ lie in some finite set depending only on $\epsilon,\Phi$.    \end{lemma}
We note that the condition of the lemma is clear satisfied if $-(K_{X'}+B'+M')$ is nef and $(X',B'+M')$ is non-exceptional. (Since abundance hold for fano type).
\newpage
%%%%%%%%%%%%%%%%
\section{Adjunction}
Here we discuss adjunction in various context. This section will form the foundation of the proof of the theorem in introduction. We will look at divisorial adjunction, adjunction for fibre space and adjunction for non-klt centres.
\subsection{Divisorial Adjunction}
\begin{defn}\label{div_adj}
Let $(X',B'+M')$ be a generalised pair with $X \xrightarrow{\phi} X'$ and M. We can assume $X$ is a log resolution of $(X',B')$. Let $S'$ be the normalisation of a component of $B'$ with coefficient 1. Then we have the following: Write $$K_X+B+M = \phi^*(K_{X'}+B'+M')$$ and define $B_S := (B-S)|_S$ and pick $M_S \sim_{\mathbb{R}} M|_S$. If $S \xrightarrow{\psi} S'$ is the induced morphism, then by letting $B_{S'} := \psi_* B_S$ and $M_{S'} := \psi_* (M_S)$, we get $$K_{S'}+B_{S'}+M_{S'} \simR (K_{X'}+B'+M')|_{S'}$$
\end{defn}
Now we will make some remarks and state some results about this adjunction.
\begin{enumerate}[wide, labelwidth=!, labelindent=10pt]
\item It is easy to see that $K_S+B_S+M_S = \psi^*(K_{S'}+B_{S'}+M_{S'})$. Furthermore, if $M=0$, we obtain the usual divisorial adjunction for pairs. If $(X',B'+M') $ is generalised lc, then $B_{S'}\in [0,1]$. Also it is clear that $M_S$ is nef on $S$, hence we get $(S',B_{S'}+M_{S'})$ is a generalised pair with data $S\xrightarrow{\psi} S'$ and $M_S$. Also $(S',B_{S'} +M_{S'})$ is generalised lc since $B_S \leq 1$.
\item \label{3.1.2} Although $B_{S'}$ is determined completely, $M_{S'}$ is only determined up to $\R-$linear equivalence. In particular if $pM$ is b-Cartier, then we can choose $M_S$ such that $pM_S$ is also b-Cartier. Therefore we can pick $M_S$ such that $$p(K_{S'}+B_{S'}+M_{S'}) \sim p(K_{X'}+B'+M')|_{S'}$$
\item We have the control on coefficients of $B_{S'}$ as usual. The proof is by some direct computation and applying the standard results on coefficients on normal divisorial adjunction.
\begin{lemma}
Let $p\in \mathbb{N}$ and $R\subset [0,1]$ be a finite set of rationals. Then there exist $\mathcal{O}\subset[0,1]$ finite set of rationals depending only on $p,R$ such that whenever $(X',B'+M')$ is generalised lc pair of dimension $d$, $B'\in \Phi(R)$ and $pM $ is b-Cartier, we have $B_{S'}\in \Phi(\mathcal{O})$  

\end{lemma}

\item Finally we have a similar inversion of adjunction.
\begin{lemma}
Using notation of \ref{div_adj}, assume further that $X'$ is $\Q$-factorial and $S'$ is a component of $B'$ with coefficient 1 and $(X',S')$ is plt. If $(S', B_{S'}+M_{S'})$ is generalised lc, then $(X',B'+M')$ is lc near $S'$.
\end{lemma}
\begin{proof}
    Assume the result doesn't hold. Firstly by replacing $B'$ by $(1-a)*S'+aB'$ and $M'$ by $\alpha M'$, we can reduce to the case that $(S', B_{S'}+M_{S'} )$ is klt. Now by letting $\Sigma := B+G+\epsilon C, \Sigma' =\phi_* \Sigma$, where $\epsilon<<1$, $space0\leq G \simR \epsilon A +M/X'$ general and $A\geq 0$ ample and $C\geq 0$ and $A+C \simR \phi^*(H)$, $H$ general very ample on $X'$, we can derive a contradiction to the standard inversion of adjunction using the pair $(X',\Sigma')$. (the idea is to add a little bit of ampleness to $M$ to make it ample over $X'$).
\end{proof}
\end{enumerate}
\subsection{Adjunction for Fibre Space}
\begin{defn}\label{fib_adj}
    Let (X,B) be a projective sub-pair and let $f: X\rightarrow Z$ be a contraction with $\dim Z>0$ such that $(X,B)$ is sub lc near the generic fibre of $F$, and $K_X+B\simR 0/Z$. We define $B_Z := \sum (1-t_D) D$ where $D$ range over all divisors on $Z$ and $t_D$ is the largest such that $(X+B+f^* D)$ is sub-lc over the generic point of $D$. Define $M_Z := L_Z -(K_Z+B_Z)$, where $K_X+B\simR f^* L_Z$. Then we clearly get (so called fibre space adjunction) $$K_X+B \simR f^*(K_Z+B_Z+M_Z)$$
\end{defn}
We will also state some properties of this adjunction. 
\begin{enumerate}[wide, labelwidth=!, labelindent=10pt]
    \item $B_Z$ is determined uniquely and $M_Z$ is only determined up to $\R$ Linear Equivalence class. 
    \item This definition is compatible with birational morphism. Assume we have $X'\xrightarrow{\phi} X, Z' \xrightarrow{\psi} Z$ where $X'\xrightarrow{f'} Z'$ is a morphism. $K_{X'}+B' =\phi^*(K_X+B)$, define $B_{Z'}$ as above for $f': X' \rightarrow Z'$ and define $M_Z' := \psi^*(L_Z)-K_{Z'}-B_{Z'}$ and $B_Z := \psi_* B_{Z'}, M_Z := \psi_*(M_{Z'})$.
    \item It is clear from definition that $M_Z$ depends only on $(X,B)$ near the generic fibre of $f$, birationally. (for precise statement see \cite{fa}, Lemma 3.5).
    \item\label{M_Z} If $(X,B)$ is lc over the generic point of $Z$, we have $M_Z$ is pseudo-effective. and if $B$ is a $\Q-$divisor, then $M_Z$ is a b-divisor (in the sense in (2)) and  $M_{Z'}$ is nef for $Z' \rightarrow Z$ sufficiently high resolution. We will omit the proof as it proof uses Hodge theory, which is not the focus of this paper. However we remark that if $(X,B)$ is lc, then $(Z,B_Z+M_Z)$ is a generalised pair. 
  \item Finally we relate singularities in the following lemma.
  \begin{lemma}\label{fib_sing}
      Suppose there is a prime divisor $S$ on some birational model of $X$ such that $a(S,X,B) \leq \epsilon$ and $S$ vertical over $Z$, then there is a resolution $Z'\rightarrow Z$ such that $B_{Z'}$ has a component $T$ with coefficient $\geq 1-\epsilon$.
  \end{lemma} The proof is obvious in the sense that for sufficient high resolution of $X'\rightarrow X$, we will have $B'$ contain $S$ as a component and $\mu_S B' \geq 1-\epsilon$, now let $T$ be the image of $S$ on $Z'$, a resolution of $Z$, then we have by definition $\mu_T B_{Z'} \geq 1-\epsilon$.
            
\end{enumerate}
Next, we will state a theorem about adjunction for fiber space which will be key for the induction treatment of complements. We will need to control the coefficients of $B_Z$ and $M_Z$ to apply induction.
\begin{prop} \label{6.3}(\cite{fa},6.3)
    Let $d\in \mathbb{N}$ and $\mathfrak{R}\subset[0,1]$ finite set of rationals. Assume Theorem \ref{com_rel} hold in dimension d. Then there exist $q\in \mathbb{N}$ and $\mathfrak{S}$, such that if $(X,B)$ is a projective lc pair of dimension $d$,  $f :X \rightarrow Z$ a contraction with $dim Z>0$, $K_X+B\simQ 0 /Z$, $B\in \Phi(\mathfrak{R})$, $X$ is of Fano type over some non-empty open subset $U\subset Z$ with generic point of all non klt centre of $(X,B)$ mapping into $U$.\\Then we have adjunction $$q(K-X+B)\sim qf^*(k_Z+B_Z+M_Z)$$ with $B_Z\in \Phi(\mathfrak{O})$ and $qM_{Z'}$ is nef Cartier for any high resolution $Z'\rightarrow Z$.
\end{prop}
\subsection{Adjunction on Non-klt Centre}
Here we talk about adjunction on non klt centres which will be very important for induction. Throughout this subsection, we fix the following \\$(X,B)$ a projective klt pair of dimension d, $G\subset X$ subvariety with normalisation $F$, X is $\Q$ factorial near the generic point of $G$, $\Delta\geq 0$ a $\R$-Cartier divisor and $(X,B+\Delta)$ is lc near generic point of G and there is a unique non-klt place of this pair whose centre is $G$.
Then we can define the following definition and proposition (see \cite{fa}, Construction 3.9).
\begin{defn}\label{non_klt}(\cite{fa},3.9 - 3.12)
    In the above setting, there exist $\Theta_F\in [0,1]$ such that $$K_F+\Theta_F+P_F \simR (K_X+B+\Delta)|_F$$ such that the following holds,
    \begin{enumerate}[wide, labelwidth=!, labelindent=10pt]
        \item $\Theta_F$ well defined and $P_F$ defined up to $\R$-linear equivalence. $P_F$ is pseudo-effective.
        \item If $B\in \Phi$, a DCC set, then $\theta_F\in \Psi$, a DCC set which only depends on $d$ and $\Phi$. Furthermore $\Phi=\Phi(\mathfrak{R})$ for some $\mathfrak{R}\subset[0,1]$ finite set, then $\Psi = \Phi(\mathfrak{S})$ for some $\mathfrak{S}\subset [0,1]$ finite set depending only on $\mathfrak{R}$.
        \item Assume $M\geq 0$ a $\Q$-Cartier divisor on $X$ with coefficients $\geq 1$ and $G\not\subset Supp M$, then for every component $D$ of $M_F := M|_F$, we have $\mu_D(\Theta_F+M_F)\geq 1$.
        \item Assume that $G$ is a general member of a covering family of subvarities, and $(X,B)$ is $\epsilon$-lc for $\epsilon>0$, then there is a boundary $B_F$ on $F$ such that $K_F+B_F = (K_X+B)|_F$ such that $(F,B_F)$ is also $\epsilon$-lc and $B_F\leq \Theta_F$.
    \end{enumerate}
\end{defn}
\proofsketch (1\&2) Let $\phi: W\rightarrow X$ be a log resolution of $(X,B+\Delta)$ that extract the place above $G$. Define $\Gamma:= (B+\Delta)^{<1}+Supp(B+\Delta)^{\geq 1}$ and $\Gamma_W$ to be the sum of $\Gamma^\sim$ and reduced exception divisor of $\phi$. Run an MMP on $(W,\Gamma_W)$ over $X$, we reach a model $(Y,\Gamma_Y)$ such that $K_Y+\Gamma_Y$ is a limit of movable divisors over $X$. Now letting $N_Y := \psi^*(K_X+B+\Delta)-(K_Y+\Gamma_Y)$. Applying negativity lemma and using similar arguments as in \ref{Q_fac}, we get that $N_Y\geq 0$, $N_Y=0$ over $U$ and $(Y,\Gamma_Y)$ is dlt and is a $\Q$ factorial dlt model for $(X,B+\Delta)$ over $U$, where $U$ is the largest open set such that $(X,B+\Delta)$ is lc , which contain the generic point of $G$. Moreover, using \cite{fa} Lemma 2.33, we see that there is a unique component $S$ of $\floor{\Gamma_Y}$ mapping onto $G$ and $h:S\rightarrow G$ is a contraction. 
Now apply divisorial adjunction on $S$ we get $K_S+\Gamma_S+N_S \simR (K_Y+\Gamma_Y+N_Y)\simR 0/F$, where $N_S := N_Y|_S$. ($S$ is not a component of $N_Y$ as $N_Y=0$ over $U$). Hence we can apply adjunction for fiber space for $h$ and get $(K_S+\Gamma_S+N_S) \simR h^*(K_F+\Delta_F+M_F)$, where $M_F$ is pseudo-effective moduli divisor (see \ref{fib_adj}). Summing up, we get $$K_F+\Delta_F+M_F \simR (K_X+B+\Delta)|_F$$ This is almost what we want except we need to control the coefficient of $\Delta_F$ only in terms of coefficient of $B$, hence we make the following modification. \\\\Let $\Sigma_Y := B^\sim + exc(\psi)\leq \Gamma_Y$ by definition. Apply divisorial adjunction, we get $K_S+\Sigma_S \simR (K_Y+\Sigma_Y)|_S$. and let $\Theta_F$ be the discriminant part of fibre space adjunction for $h$ for the pair $(S,\Sigma_S)$. Also let $P_F$ be such that $K_F+\Theta_F+P_F \simR (K_X+B+\Delta)|_F$. Then since $\Sigma_S\leq \Gamma_S$, we have $\Theta_F\leq \Delta_F$, hence $P_F\simR \Delta_F-\Theta_F+M_F$ which is pseudo-effective, which shows (1). Now (2) follows from we have coefficients of $\Sigma_S$ belong to some DCC set depending only on $\Phi$, and using ACC of lct we get that coefficients of $\Theta_F$ belong to some fixed DCC set. The second part of (2) follows from (\cite{fa}, 6.3).\\
(3) consider $\Sigma_Y' := \Sigma_Y+M_Y$, where $M_Y := \psi^* M$ and $K_S+\Sigma_S' := (K_Y+\Sigma_Y')|_S$. Define $\Theta_F '$ for replacing $\Sigma_Y$ with $\Sigma_Y'$ in above construction. We see that $\mu_D(\theta_F+M_F) = \mu_D(\Theta_F')$. However coefficients of $M\geq 1$, hence if $D$ is a component of $M_F$, then $h^* D$ is a component of $\floor{\Sigma_S'}$. Therefore, we see that $lct(h^*D,S,\Sigma_S')\leq 0$ (lct is over the generic point of $D$), hence $\mu_D(\Theta_F')\geq 1$ as claimed.
\\\textbf{(4)} is long so we just briefly sketch it. (Details see \cite{fa},3.12) Since $G$ belongs to a bounded family of covering subvarieties, by \ref{bnd_fam}, we can assume $G$, appears as a fibre of $V\rightarrow T$ such that the morphism $V\rightarrow X$ is surjective. By using some easy modification(take normalisation, then log resolution of $V$ and $T$, and then cut the base $T$ by hyper-surface sections) and using that $G$ is general, we obtain $F'$ appear as a general fibre of $W\rightarrow R$ such that $W\xrightarrow{\phi} X$ is surjective, generically finite and etale over $\phi(\eta_{F'})$ and $F'\xrightarrow{\phi|_{F'}} G$ is just $F'\xrightarrow{resolution} F\xrightarrow{normalisation} G$. Furthermore, by taking higher resolution we can assume $Q_{W'}=supp\phi^*(P)$ is relatively simple normal crossing over some non empty open set of $R'$, where P is Cartier on X such that $0\geq P\geq Supp B+X_{Sing}$. In particular by generality of $G$, we have $Q_{F'} := Q_{W'}|_{F'}$ is a reduced snc divisor. Now let $W'\rightarrow W\rightarrow X$ to be the stein factorization for $\phi$, we get $W'\rightarrow W$ is birational and $W\rightarrow X$ is finite. Hence by (\cite{km}, 5.20), we get $W,B_W$ sub $\epsilon$-lc hence $W',B_W'$ sub $\epsilon$-lc, where $K_W+B_W := (K_X+B)|_W$ and $K_{W'}+B_{W'} := (K_W+B_W)|_{W'}$. Finally define $B_F' := B_W'|_{F'}$ (note $K_{W'}|_{F'}=K_{F'}$ by assumption) and $B_F$ is the push forward of $B_{F'}$ to $F$. Hence we get $K_F+B_F =(K_X+B)|_F$. Finally we note that $(W',B_W')$ sub $\epsilon$-lc, which implies $(W',B_W'^{\geq 0} )$ $\epsilon$-lc, which implies $(F',B_F'^{\geq 0} )$ $\epsilon$-lc (since $supp(B_W'^{\geq 0})\subset Q_{W'}$ by construction and $Q_{W'}$ reduced snc near $F'$), which implies $(F,B_{F})$ sub $\epsilon$-lc as claimed. Finally to show $B_F\leq \Theta_F$ is just technical so we omit it here.\qed 

\subsection{Lifting Section from Non-klt Centres}
Here we introduce results about lifting sections from non-klt centre which is an important ingredient for proof in the next section.
\begin{prop}(\cite{fa},3.15)\label{lif_sec}
     For $d,r\in \N$ and $\epsilon\in \R^{>0}$, there exist $l\in \N$ depending only on $d,r,\epsilon$ such that if assume notation and set-up in 3.7 (4) and assume further that $X$ is Fano of dimension $d$ and $B=0$, $\Delta \simQ -(n+1)K_X$ for some $n\in \N$, $h^0(-nrK_X|_F)\neq 0$, $P_F$ is big and for any choice of $ P_F\geq 0$, $(F,\Theta_F+P_F)$ is $\epsilon$-lc,then $h^0(-lnrK_X)\neq 0$ 
\end{prop}
We first prove a lemma, which will make the proof a lot cleaner.
\begin{lemma}\label{lemma_lift_sec}
    Assume notation in 3.7 (4), and $P_F$ is big, then if there exist $D$ divisors on birational model of S such that $a(D,S,\Delta_S := \Gamma_S+N_S)<\epsilon$ and centre of $D$ on S is vertical over $F$, then we can choose $P_F$ such that $(F,\Theta_F+P_F)$ is not $\epsilon$-lc.
\end{lemma}
\begin{proof}
    Recall we had $(K_S+\Gamma_S+N_S) \simR h^*(K_F+\Delta_F+M_F)$ where $\Delta_F$ is the discriminant part and $M_F$ is the moduli part, and its clear from construction that $\Delta_F\geq \Theta_F)$. Applying Lemma \ref{fib_sing}, we see that there is a resolution $F'\rightarrow F$ such that $\Delta_{F'}$ will have a component $T$ with coefficient $\geq (1-\epsilon)$. We can see from definition that $P_F\simR \Delta_F+M_F-\Theta_F$ and $P_F\simR A+C$, where $A$ is ample and $C\geq 0$ as $P_F$ is big. The idea of the proof is that if $\Delta_F+M_F-\Theta_F$ is effective then we are clearly done. If it is not effective, we can use the ampleness of $A$ to make it effective. The rigours proof is as following. \\We denote $K_{F'}+D := (K_F+\Theta_F+C)|_{F'} $ and let $A' := A|_{F'}$ which is nef and big where $F'\rightarrow F$ is a high resolution. Now consider $t>0$ sufficiently small and pick $0\geq L \simR t A'+(1-t)M_{F'}$, then we get $K_{F'}+\Omega':=K_{F'}+tD +L+(1-t)(\Delta_{F'}) \simR t(K_F+\Theta_F+P_F)|_{F'}+ (1-t)(K_F+\Delta_F+M_F)|_{F'}\simR (K_F+\Theta_F+P_F)|_{F'}$, which is trivial over $F$. Therefore it is the pullback of its push down to $F$, say $K_F+\Omega$, by negativity lemma. Hence we get $\Omega\simR \Theta_F+P_F$.Also note that $(F',\Omega')$ is not sub $\epsilon$-lc, hence $(F,\Omega)$ is not $\epsilon$-lc. Its clear that $\Omega\geq \Theta_F$ since $\Delta_F\geq \Theta_F$. Therefore we are done by replacing $P_F $ by $\Omega-\Theta_F$.
\end{proof}

\begin{proof}[Proof of Prop \ref{lif_sec}]
    Assuming notation in Definition \ref{non_klt}. We have $\psi: Y\rightarrow X$ and $S\in\floor{\Gamma_Y}$ the unique component mapping to $G$. Also let $K_Y+\Delta_Y := \psi^*(K_X+\Delta)$ and $k_S+\Delta_S := (K_Y+\Delta_Y)|_S$. Note that we have $\Delta_Y =\Gamma_Y+N_Y$ and $Supp(\Delta_Y) = Supp(\Gamma_Y)$\\Firstly, $G$ is an isolated non-klt centre and no components of $\floor{\Delta_Y-S}$ intersect $S$: Indeed if not,say $T$ is a component of $\floor{\Delta_Y-S}$ intersecting $S$,(hence centre of $T$ on $X$ is non klt centre intersecting $G$), then there will be a component $T_S$ of $\floor{\Delta_S}$ which is contradicts the Lemma \ref{lemma_lift_sec}. Also we can assume that $E := \psi^*(-nrK_X)$ is an integral divisor near $S$ since $G$ is general and we can assume $K_X$ and $E$ has bounded Cartier index depending only on $\epsilon$ near codimensional one points of S.\\Claim: $h^1(\ceiling{lE-\floor{\Gamma_Y}-N_Y})=0$ for any $l\geq 2$. Define $L := \ceiling{lE-\floor{\Gamma_Y}-N_Y} -(lE-\floor{\Gamma_Y}-N_Y) \geq 0$, then we see that by above assumption, $S$ is not a component of L and every components of $L$ is either exceptional over $X$ or supported in $N_Y$, which means $L$ is supported in $\floor{\Gamma_Y}$ with coefficients $\leq 1$. Therefore we see that $(Y,\Gamma_Y-\floor{\Gamma_Y}+L)$ is klt since $(Y,\Gamma_Y)$ is dlt. Finally observe we have $$\ceiling{lE-\floor{\Gamma_Y}-N_Y}\simQ K_Y+\Gamma_Y-\floor{\Gamma_Y}+L+ (-lnr+n)\psi^* K_X$$ and $\psi^*(K_X)$ is nef and big, therefore Kawamata Viehweg vanishing, we get the claim.\\Using Lemma 2.42 in \cite{fa}, the fact that Cartier Index of $E$ is bounded near on codimensional one points of $S$ and $\floor{\Gamma_Y}+N_Y-S$ doesn't intersect $S$, we can pick a bounded $l$ such that we following sequence is exact. $$0\rightarrow \OO_Y(\ceiling{lE-\floor{\Gamma_Y}-N_Y})\rightarrow \OO_Y(\ceiling{lE-\floor{\Gamma_Y}-N_Y+S})\rightarrow \OO_Y(\ceiling{lE}|_S)\rightarrow 0$$ Combining with the claim above, we get $$h^0(\ceiling{lE-\floor{\Gamma_Y}-N_Y+S})\twoheadrightarrow h^0(\ceiling{lE}|_S)$$Since $\ceiling{lE}|_S = lE|_S = h^*(-lnrK_X|_F)$ where $h:S\rightarrow F$, hence we get $h^0(\ceiling{lE-\floor{\Gamma_Y}-N_Y+S})\neq 0$, which implies $h^0(\ceiling{lE})\neq 0$ (as $-\floor{\Gamma_Y}-N_Y+S\leq 0$), therefore we get $h^0(-lnrK_X)\neq 0$. 
\end{proof}

\newpage
%%%%%%%%%%%%%%%%
\section{Effective Birationality}
Now we will look at the first key results of the paper. The next proposition is the main tool where we will use it to derive contradictions. Although we will prove a much stronger result in theorem \ref{1.6}, but we still need this version because the proof of \ref{1.6} is based on the following lemma.
\subsection{Boundedness of Singularities in Bounded Family}
\begin{prop}\label{4.2}(\cite{fa} 4.2)
    Let $\epsilon\in \R^{\geq 0}$ and $\mathcal{P}$ be a bounded set of couples. Then there is $\delta\geq 0$ depending only on $\epsilon,\mathcal{P}$ such that if $(X,B)$ is  $\epsilon$-lc and $(X,Supp(B)+T)\in \mathcal{P}$ for some reduced divisor T, $L\geq 0$ is an $\R$-Cartier $\R$ divisor, where $L\simR N$ for some $N$ supported in $T$, and $N\in [-\delta,\delta]$, then we have $(X,B+L)$ is klt.
    
\end{prop}
\begin{proof}
    We will sketch the proof here. By induction on $d$, we may assume all varieties in $\mathcal{P}$ have dimension $d$. (note the proposition is clear if $d=1$, we wlog $d\geq 2$). Assume the claim is false. Let $\phi: W \rightarrow X$ be a log resolution of $(X,Supp(B)+T)$ and let $K_W+B' := \phi^*(K_X+B)$, let $B_W := B'^{\geq 0}$ and $N_W := \phi^*N$, $L_W := \phi^*L$ and $T_W := \phi^* T$. Since $X$ is bounded, we can assume that $N_W$ has coefficients with absolute value less than $n\delta$, where $n$ depends only on $\mathcal{P}$. Using \ref{bnd_fam}, we can replace $X,B,N,L,T$ with $W,B_W,N_W,L_W,T_W$ and from now on we will assume $(X,Supp(B0+T))$ is log smooth. Furthermore, by using \ref{bnd_fam} again, we can add a large multiple (but only depending on $\mathcal{P}$) of a very ample divisor on $X$ to $T$ and hence we can assume $T$ is very ample  (we need to modify $\mathcal{P}$ accordingly).\\Using boundedness of $T$, we see that $T^d$ is bounded above depending only on $\mathcal{P}$, say $T^d\leq M$, then we claim we can take any $\delta < \epsilon/M$: Assume not,  then we have $(X,B+L)$ is not klt, since $(X,B)$ is log smooth and $\epsilon$-lc, there exist $x\in X$ such that $mult_x(L)\geq \epsilon$, but $mult_x(L)\leq L T^{d-1}\leq N*T^{d-1}\leq \delta T^{d}< \epsilon$, which is a contradiction.
\end{proof}
\subsection{From Boundedness of Volume to Bounded Family}
In order to apply the above lemma, we need to have a method to create some bounded family. The following is a standard lemma to create bounded families. We omit its proof here.
\begin{lemma}(\cite{BB},2.4.2, 3.2)\label{4.4.1}
    Let $V_1,V_2$ be 2 fixed positive real number. Let $(X,D)$ be a couple where $X$ is a normal projective variety of dimension $d$. Suppose we have $A$, a base point free Cartier divisor such that $|A|$ define a birational morphism $\phi_A : X\rightarrow X'$ and $vol(A)\leq V_1$. Assume further that $vol(K_X+D+2(d+1)A)\leq V_2$ is bounded above. Then all such $(X,D)$ form a log birationally bounded family and $(X',D')$, where $D' := \phi_* D$, form a log bounded family. Also by \ref{bnd_fam}, the set of $(\bar{X},\bar{D})$ form a log smooth log bounded family, where $\bar{X}$ is a log resolution of $(X',D')$ and $\bar{D} := D'^{\sim}$+ reduced exceptional divisor over $X'$. 
\end{lemma}
    \begin{rmk}\label{4.4.2}
    Also in order to apply \ref{4.2}, we need to bound the coefficients in bounded family. Assume that $(X,D)\in \mathcal{P}$, a bounded family and assume that we have $A$, a big divisor, with $Supp A\leq D$ and we have $M\geq A$ a divisor with bounded volume, say $c$. Then coefficients of $M$ are bounded above depending only on $\mathcal{P}$ and $c$: Indeed, since $A\leq D$, we can choose a $l\in \N$ and $H$ very ample on $X$ depending only on $\mathcal{P}$, such that $lA-H$ is big. Then $$M\cdot H^{d-1}\leq vol(M+H)\leq vol(M+lA)\leq vol((l+1)M)$$which implies coefficients of M are bounded above.
\end{rmk}
\subsection{Main Result}
The main result of this section is the following proposition. Observe this statement together with theorem \ref{com_gen} trivially implies effective birationality (Theorem \ref{eff_bir}).
\begin{prop}\label{4.9}(\cite{fa},4.9)
    Let $d\in \N$ and $\epsilon,\delta\in\R^{>0}$, then there exist $m\in \N$ depending only on $d,\epsilon,\delta$ such that if $X$ is $\epsilon$-lc Fano of dimension d and there exist $B\geq \delta$, $\Q$-divisor such that $K_X+B\simQ 0$, then $|-mK_X|$ defines a birational map.
\end{prop}
    Here we say a few works about the proof. Firstly, the idea is to create some bounded family using \ref{4.4.1} and then apply \ref{4.2}. However, this in term requires some boundedness statement of the volume $vol(-mK_X)$. The next lemma gives us the tool to bound the anticanonical volume.
    
\begin{lemma}\label{4.7}(\cite{fa},4.7)
    Let $d\in \N$ and $\epsilon,\delta\in\R^{>0}$, then there exist $p\in \N$ depending only on $d,\epsilon,\delta$ such that if $X$ is $\epsilon$-lc Fano of dimension d, m is smallest integer such that $|-mK_X|$ defines a birational map, $n$ smallest integer such that $vol(-nK_X)>(2d)^d$ and $nK_X+N\simQ 0$ for some $\Q$-divisors $N\geq \delta$, then $\frac{m}{n}<p$
\end{lemma}
We will prove the proposition now given Lemma \ref{4.7}.
\begin{proof}[Proof of Prop \ref{4.9}] 
    Step 1: Assume the claim is false, then we get a sequence of $X_i,B_i,m_i$ as in the proposition such that $m_i$ is the smallest integer such that $|-m_i K_{X_i}|$ defines a birational map and $m_i\rightarrow \infty$. Also choose $n_i$ smallest integer such that $vol(-n_i K_{X_i})>(2d)^d\geq vol(-(n_i-1) K_{X_i})$, then we have by Lemma \ref{4.7}, there is a fixed p, such that $\frac{m_i}{n_i}<p$ for all $i$. Hence $n_i\rightarrow \infty$ and we can assume $\frac{m_i}{n_i-1}\leq 2p$. Therefore $vol(-m_i K_{X_i})$ is bounded uniformly from above. Choose $0\leq M_i\in |-m_i\cdot K_{X_i}|$ and let $\phi : W_i\rightarrow X_i$ be log resolution such that $M_i' := \phi^*(M_i) :=  A_i'+R_i'$, where $A_i'$ is a general element in the movable part which is  base point free and it defines a birational map and $R_i\geq 0$ is the fixed part (by \ref{2.6}). Let $A_i$ and $R_i$ be their  pushdown to $X_i$, note that $A_i$ and $R_i$ are both integral divisors.\\
    Step 2: (Construct log bounded family) Now let $\Omega_{W_i}$ to be the sum of birational transform of $M_i$ and reduced exception divisor of $\phi$, and let $\Sigma_{W_i}:= Supp(\Omega_{W_i})$, then we have $\Sigma_i := \phi_*(\Sigma_{W_i}) \leq M_i$. Then we have $$Vol(K_{W_i}+\Sigma_{W_i}+2(2d+1)A_i')\leq vol(K_{X_i}+(4d+2)A_i+M_i)\leq vol((4d+3)M_i)$$ which is bounded uniformly. Also $vol(A_i')\leq vol(M_i')$ is bounded. Therefore using \ref{4.4.1}, we see that if $\psi_i : W_i\rightarrow \hat{W_i}$ is the birational contraction defined by $|A_i'|$, and $\Sigma_{\hat{W_i}}$ is the pushforward of $\Sigma_{W_i}$, then $(\hat{W_i},\Sigma_{\hat{W_i}})$ is log bounded. Taking log resolution $\bar{W_i}\rightarrow \hat{W_i}$ and letting $\Sigma_{\bar{W_i}}$ to be sum of the birational transform of $\Sigma_{\hat{W_i}}$ and the reduced exceptional divisors, we see that $(\bar{W_i},\Sigma_{\bar{W_i}})$ is log smooth log bounded. Replacing $W_i$ we can assume $h_i :W_i\rightarrow \bar{W_i}$ is a morphism and let $A_{\bar{W_i}},M_{\bar{W_i}}$ be the pushforward of $A_i',M_i'$ to $\bar{W_i}$. Applying \ref{4.4.2}, we see that coefficients of $M_{\bar{W_i}}$ are bounded above since $Supp(M_{\bar{W_i}})\subset \Sigma_{\bar{W_i}}$.Also note that by construction $\Sigma_{\bar{W_i}}$ contains the support of $M_{\bar{W_i}}$ and all exceptional divisors of $\bar{W_i}\dashrightarrow X_i$.We note here that we can also use a much weaker version of \ref{4.4.3} here.\\
    Step 3: Now we are ready to derive contradiction using \ref{4.2}. Let $K_{W_i}+\Lambda_{W_i} := K_{X_i}|_{W_i}$ where $\Lambda_{W_i}\leq 1-\epsilon$ as $X_i$ is $\epsilon$-lc. Now since $K_{X_i}+\frac{1}{m_i}M_i\simQ 0$, Let $K_{\bar{W_i}}+\Lambda_{W_i}+\frac{1}{m_i}M_{\bar{W_i}}\simQ 0$ be its crepant pullback to $\bar{W_i}$. Now since $0\leq \frac{1}{m_i}M_{\bar{W_i}}\rightarrow 0$ and supported in $\Sigma_{\bar{W_i}}$, we have $(\bar{W_i},\Lambda_{W_i}^{\geq 0}+\frac{1}{m_i}M_{\bar{W_i}} ) $ is $\epsilon/2$-lc for all $i>>0$. However $K_{X_i}+\frac{1}{m_i}M_i+\frac{1}{\delta}B_i$ is ample and $(X_i,\frac{1}{m_i}M_i+\frac{1}{\delta}B_i)$ not klt, hence by negativity, $(\bar{W_i},\Lambda_{W_i}+\frac{1}{m_i}M_{\bar{W_i}}+\frac{1}{\delta}B_{\bar{W_i}})$ is not sub-klt, where $B_{\bar{W_i}}$ is pushforward of $B_i|_{W_i}$, which implies $(\bar{W_i},\Lambda_{W_i}^{\geq 0}+\frac{1}{m_i}M_{\bar{W_i}}+\frac{1}{\delta}B_{\bar{W_i}}) $ is not klt, this contradicts \ref{4.2} since $\frac{1}{\delta}B_{\bar{W_i}}\simQ \frac{1}{\delta m_i} M_{\bar{W_i}}$.
\end{proof}
  Now we will give the proof of lemma \ref{4.7}. The idea is to create and non-klt centres and derive a contradiction. First, we need a subtle lemma in order to get a bounded family.
  \begin{lemma}\label{4.4.3}
      Let $d\in\N$ and $\epsilon,v>0$. Then there exist log bound family $\mathcal{P}$ and $c>0$ depending only on $d,\epsilon,v$ such that: Suppose $X$ is normal projective of dimension $d$ and $B\geq \epsilon$, $M\geq 0$ nef $\Q$-divisor such that $|M|$ defines a birational map, $M-(K_X+B)$is pseudoeffective, $vol(M)<v$ and $\mu_D(B+M)>1 $ for all $D$ component of $M$. \\Then there exist $(\bar{X},\Sigma_{\bar{X}})\in \mathcal{P}$ such that 
      \begin{enumerate}
          \item $\bar{X}\rightarrow X$ is birational and $\Sigma_{\bar{X}}$ contain all the exception divisors and support of birational transform of $(B+M)$
          \item there is $X'\rightarrow X,X'\rightarrow \bar{X}$ is a common resolution and coefficients of $M_{\bar{X}}$ are less than $c$, where $M_{\bar{X}}$ is pushdown of $M':= M|_{X'}$. Also $M'\simQ A'+R'$, where $|A'|$ is the fixed part of $M'$ and is base point free, and $A'\simQ 0/\bar{X}$
      \end{enumerate}  
      
  \end{lemma} 
\begin{proof}[Proof of Lemma \ref{4.4.3}]
    Step 1: We first set up standard notation. we can wlog $\epsilon<\frac{1}{2}$. Since $|M|$ defines birational map, $M$ is big. Let $X'\rightarrow X$ be a log resolution of $(X,B+M)$, such that  $M' := M|_{X'}\simQ A'+R'$, where $|A'|$ is the movable part and $|A'|$ base point free defining a birational contraction and $R'$ is the fixed part. Let $A,R$ be their pushdown to $X$. Note $A'$ is nef and big. \\
    Step 2: Let $H\in |6dA'|$, be general, Now we define a boundary $\displaystyle \Omega_{X'}:= \frac{1}{2}\sum_{D\in I_1}D+\epsilon\sum_{D\in I2}D +\frac{1}{2}H$, where $I_1 =exc/X+supp(M')$ and $I_2$ contain all components in $B^\sim$ but not in $M'$. Then we have $(X',\Omega_{X'})$ is $\epsilon$-lc, log smooth, and $K_{X'}+\Omega_{X'} = (K_X+\frac{1}{2}\sum_{D\in I_1}D+\epsilon\sum_{D\in I2}D)+\frac{1}{2}H$ is big by \ref{2.46}. \\
    Step 3: We claim that $vol(K_{X'}+\Omega_{X'})$ is bounded above. Indeed let $\Omega_X$ be the pushdown of $\Omega_{X'}$, then $vol(K_X+\Omega_X)\leq vol(K_X+B+5dM)\leq vol(6dM)$ is bounded, where is the first inequality is because $5dM+B-\Omega_X = (B+M+\frac{1}{2}H-\Omega_X)+(4dM-\frac{1}{2}H)$ is big (since $B+M+\frac{1}{2}H-\Omega_X\geq 0$ by assumption and $4dM-\frac{1}{2}H\simQ 4dM-3dA$ is big). \\
    Step 4: We finish the proof now. Let $\Sigma_{X'}:= Supp(\Omega_{X'})$. Also note $K_{X'}+\Omega_{X'}$ is big and coefficient of $\Omega_{X'}\in \{\frac{1}{2},\epsilon\}$, so there exist $\alpha\in(0,1)$ depending only on $\epsilon$ such that $K_{X'}+\alpha\Omega_{X'}$ is big. Now letting $p$ be large but bounded above we see that  \begin{multline*}
    vol(K_{X'}+\Sigma_{X'}+2(2d+1)A')\leq vol(K_{X'}+p(1-\alpha)\Sigma_{X'})\leq vol(K_{X'}+p(1-\alpha)\Sigma_{X'}+p(K_{X'}+\alpha\Sigma_{X'}))\\\leq vol((1+p)(K_{X'}+\Omega_{X'}))
    \end{multline*}
    Now the remaining of the proof is clear by applying Lemma \ref{4.4.1} and Remark \ref{4.4.2}.
    
\end{proof}
  \begin{proof}[Proof of Lemma \ref{4.7}]
      We first assume the theorem is false. Then we have a sequence of $X_i,m_i,n_i,B_i$ as in the lemma such that $\frac{m_i}{n_i}\rightarrow \infty$. \\
      Step 1: We create some non-klt centre $G_i$ with positive dimension and $vol(-m_iK_{X_i}|_{G_i})$ is bounded. Fix i for now. Using [\cite{fa},2.32(2)], we see that there is a covering family of subvarieties $G_i$ such that, for any general 2 points $x,y\in X_i$, there exist $0\leq \Delta_i\simQ (-nK_{X_i})$ such that $(X_i,\Delta_i)$ is lc but not klt at ${x}$, with a unique non klt place above a non-klt centre $G_i$ containing ${x}$ and not lc at $y$. Now if $dim(G_i)=0$ for a covering family of $G_i$, then we have $-2n_i K_{X_i}$ is potentially birational, hence $m_i\leq 2n_i$ a contradiction. Hence by passing to a subsequence of $X_i$ if needed, we can assume general $G_i$ has dimension $>0$. Now let $l_i$ be the minimum integer such that $vol(-l_i K_{X_i}|_{G_i})>d^d$. If $\frac{l_i}{n_i}<a$ is bounded above, by [\cite{fa},2.32(2)], we can replace $n_i$ with $dan_i$ and hence reducing dimension of a general $G_i$. This can't happen indefinitely else we end up with potentially birational again and a contradiction. Hence we can assume $\frac{\l_i}{n_i}\rightarrow \infty$. Now if $\frac{m_i}{l_i}\rightarrow \infty$, then we can replace $n_i$ by $l_i$ and reduce dimension of $G_i$, again. Repeating the above argument, we can assume $\frac{m_i}{l_i}$ is bounded uniformly above, and hence $vol(-m_iK_{X_i}|_{G_i})\leq \frac{m_i}{l_i-1}vol(-(l_i-1)K_{X_i}|_{G_i})$ is uniformly bounded. From now on, for each $X_i$, we assume we have a such $G_i$ with positive dimension and $vol(-m_iK_{X_i}|_{G_i})$ is bounded. \\
      Step 2: Consider adjunction on non-klt centre (\ref{non_klt}) of the pair $(X_i,\Delta_i)$, we get $K_{X_i}+\theta_{F_i}+ P_{F_i} \simQ (K_{X_i}+\Delta_i)|_{F_i}$, where $F_i$ is normalisation of $G_i$ and $X_i$ is lc near generic point of $G_i$ as $G_i$ general, also by replacing $n_i$ with $2n_i$ and adding general $0\leq Q_i \simQ -n_iK_{X_i}$ to $\Delta_i$ (This doesn't change $\Delta_i$), and replacing $P_{F_i}$ with $P_{F_i}+Q_i|_{F_i}$ and make it effective in its $\Q$ linear class, we can assume $P_{F_i}$ is big and effective. Lets choose $0\leq M_i\in |-m_iK_{X_i}|$ and let $\phi : W_i\rightarrow X_i$ be log resolution such that $M_i' := \phi^*(M_i) :=  A_i'+R_i'$, where $A_i'$ is a general element in the movable part which is  base point free and it defines a birational map and $R_i\geq 0$ is the fixed part. Let $A_i$ and $R_i$ be their  pushdown to $X_i$, note that $A_i$ and $R_i$ are both integral divisors. Let $F_i' \rightarrow F_i$ be log resolution of $(F_i,M_{F_i}+\theta_{F_i})$. Also let $M_{F_i'} := M_i'|_{F_i'}, A_{F_i'} := A_i'|_{F_i'}$ and $A_{F_i}$ be the pushdown of $A_{F_i'}$. We see that $vol(M_{F_i})$ is bounded and $|A_{F_i}|$ defines a birational map since $G_i$ is general.  Now since $\theta_{F_i}$ has coefficient belong to some fixed DCC set $\Phi$ depending only on $d$ (\ref{non_klt}), pick $\epsilon'\leq min(\Phi)^{>0}$, then we have $\theta_{F_i}\geq \epsilon'$. Also by \ref{non_klt},we have $\mu_D(\theta_{F_i}+M_{F_i})>1$ for any component $D$ of $M_{F_i}$. Therefore we can apply lemma \ref{4.4.3}, and we see that $(F_i',\Sigma_{F_i'})$ is log birationally bounded, where $\Sigma_{F_i'} := Supp(\theta_{F_i}+M_{F_i})^\sim+exc/F_i$. Similarly we have a log smooth log bounded family $(\bar{F_i},\Sigma_{\bar{F_i}})$ where $\Sigma_{\bar{F_i}}$ contain the support of birational transform of $\theta_{F_i}+M_{F_i}$ and all exceptional divisor of $\bar{F_i}\dashrightarrow F_i$. Also similarly, we get coefficient of $M_{\bar{F_i}}$ is bounded above, where $M_{\bar{F_i}}$ is the pushforward of $M_{F_i'}$ to $\bar{F_i}$\\
      Step 3: Now we will finish the proof as in Step 3 for proof of \ref{4.9}. Let $K_{F_i}+\Lambda_{F_i} := K_{X_i}|_{F_i}$, by \ref{non_klt} (4), we see that $(F_i,\Lambda_{F_i})$ is sub $\epsilon$-lc and $\Lambda_{F_i}\leq \theta_{F_i}$. Define $\Gamma_{\bar{F_i}} := (1-\epsilon) \Sigma_{\bar{F_i}}$. Let $K_{F'}+\Lambda_{F_i'} := (K_{F_i}+\Lambda_{F_i})|_{F_i'}$ and $K_{\bar{F_i}}+\Lambda_{\bar{F_i}}$ be its pushdown to $\bar{F_i}$. Then it is clear that $\Lambda_{\bar{F_i}}\leq \Gamma_{\bar{F_i}}$. Let $I_{F_i}:= \theta_{F_i}+P_{F_i}-\Lambda_{F_i}\geq 0$. Let $I_{\bar{F_i}}$ be the pushdown of $I_{F_i}|_{F_i'}$. we see that $I_{\bar{F_i}}\geq 0$ and $I_{F_i}\simR \frac{n_i+1}{m_i}M_{F_i}$ by definition, hence we get $I_{\bar{F_i}}\simR \frac{n_i+1}{m_i}M_{\bar{F_i}}$.\\
      Step 4: we are ready to derive contradiction. Let $N_{F_i} := \frac{1}{\delta}N_i|_{F_i}$, which has coefficient $>1$. Define $N_{\bar{F_i}}$ in the usual way. Note $N_i\simQ \frac{n_i}{\delta m_i}M_i$. Then we have $(F_i,\Lambda_{F_i}+I_{F_i}+N_{F_i})$ not klt. Then since  $K_{F_i}+\Lambda_{F_i}+I_{F_i}+N_{F_i}=K_{F_i}+\theta_{F_i}+P_{F_i}+N_{F_i}$ is ample, hence by negativity lemma, $(\bar{F_i},\Lambda_{\bar{F_i}}+I_{\bar{F_i}}+N_{\bar{F_i}})$ is not sub-klt. Hence $(\bar{F_i},\Gamma_{\bar{F_i}}+I_{\bar{F_i}}+N_{\bar{F_i}})$ is not klt. But $(\bar{F_i},\Gamma_{\bar{F_i}})$ is $\epsilon$-lc. Also $I_{\bar{F_i}}+N_{\bar{F_i}}\simQ \frac{(\delta(n_i+1)+1}{\delta m_i}M_{\bar{F_i}}$, which tends towards $0$ since coefficient of $M_{\bar{F_i}}$ is bounded. This contradicts \ref{4.2}. 
  \end{proof}
Here we state one more result which will be important for induction treatment of complements. 
\begin{prop}\label{4.11}(\cite{fa}, 4.11)
    Let $d\in \N$, then there exist $m\in \N,\epsilon>0$ depending only on $d$ such that if $X$ is a $\epsilon$-lc fano variety of dimension $d$, then $|-mK_X|$ defines a birational map. 
\end{prop}
We will outline and sketch the proof. It is very similar to the proof of \ref{4.7} and \ref{4.9}. We will focus on the difference only and try to use the same notation.
\begin{proof}
    Assume the claim is false. then we have $X_i,\epsilon_i,m_i$ such that $\epsilon_i\rightarrow 1$ and $m_i\rightarrow \infty$. let $n_i$ be the smallest integer such that $vol(-n_i K_{X_i})>(2d)^d$. Also let $\Delta_i,M_i,M_{F_i}$ as in the proof of \ref{4.9}.
    Step 1: We firstly claim that it suffice to show $\frac{m_i}{n_i}$ is bounded from above: Indeed if  $\frac{m_i}{n_i}$ is bounded from above, then $vol(-m_iK_{X_i})$ is bounded, then we can just repeat the proof \ref{4.9}, and get (in the same notation) $K_{\bar{W_i}}+\Lambda_{\bar{W_i}}+\frac{1}{m_i}M_{\bar{W_i}}\simQ 0$, but now $\Lambda_{\bar{W_i}}\leq 1-\epsilon_i$ and $\frac{1}{m_i}M_{\bar{W_i}}\rightarrow 0$ as coefficients of $M_{\bar{W_i}}$ bounded from above. Hence we get $K_{\bar{W_i}}$ is pseudoeffective, which implies $K_{X_i}$ is pseudo effective, which is contradiction as $-K_{X_i}$ is ample. So it suffice to show $\frac{m_i}{n_i}$ is bounded. \\
    Step 2: We follow the proof of \ref{4.7}. Using the same notation, we see that $(F_i,\theta_{F_i})$ is $\epsilon-lc$ with  $\epsilon\rightarrow 1$ as $i \rightarrow \infty$. However coefficient of $\theta_{F_i}$ belongs to a fixed DCC set independent of $i$. Hence we have $\theta_{F_i}=0$ for all $i$ sufficiently large, and $K_{F_i}+P_{F_i}\simQ (K_{X_i}+\Delta_i)|_{F_i}$, where $0\leq \Delta_i\simQ (-n_i-1)K_{X_i}$ and $P_{F_i}$ big and effective by construction in the proof of \ref{4.7}. Now we apply Lemma \ref{4.4.3} on $F_i$ with $X=F_i,B=0,M=M_{F_i}$,(note $M_{F_i}-K_{F_i}$ is big by construction), we get a bounded family $\mathcal{P}$ and  $(\bar{F_i},\Sigma_{\bar{F_i}})\in \mathcal{P}$ that satisfies condition 1 and 2 as in Lemma \ref{4.4.3}. Now letting $K_{X_i}+\Lambda_{F_i} := K_{X_i}|_{F_i}$, by Lemma \ref{non_klt}, $\Lambda_{F_i}\leq \theta_{F_i}=0$ and $(F_i,\Lambda_{F_i})$ sub-$\epsilon_i$-lc. Now as in proof of \ref{4.9}, we have $K_{F_i}+\Lambda_{F_i}+\frac{1}{m_i}M_{F_i}\simQ 0$, taking crepant pullback to $\bar{F_i}$, we get $K_{\bar{F_i}}+\Lambda_{\bar{F_i}}+\frac{1}{m_i}M_{\bar{F_i}}\simQ 0$, and as in step 1, we see that $K_{\bar{F_i}}$ is pseudoeffective. \\
    Step 3: Now we finish the proof in the case of $k_{\sigma}(K_{\bar{F_i}})>0$. By adding some very ample divisor to $\Sigma_{\bar{F_i}}$, we can assume there is $0\leq H_i\leq \Sigma_{F_i'}$ very ample, and there is $l\in \N$ independent of $i$, such that $lA_{\bar{F_i}}-H_i$ is big. Then by \ref{2.38}, $vol(pK_{\bar{F_i}}+H_i)\rightarrow \infty$ uniformly as $p\rightarrow \infty$. This implies $vol(-m_i(1+l)K_{X_i}|_{F_i})\geq vol(\frac{m_i}{n_i}(K_{F_i}+P_{F_i})+lA_{F_i})\rightarrow \infty$, which is contradiction, since we assumed $vol(-m_iK_{X_i}|_{F_i})$ is bounded above, (see proof lemma \ref{4.7}, step 1).\\
    Step 4: Now we deal with $k_{\sigma}(K_{\bar{F_i}})=0$, by  \ref{2.38}, there eixst bounded $r\in \N $ such that $h^0(rK_{\bar{F_i}})\neq 0$, which implies $h^0(rK_{F_i})\neq 0$, which means there is  integral divisor $0\leq T_{F_i}\sim rK_{F_i}$. We can derive contradiction similar to proof of \ref{4.9} if $T_{F_i}\neq 0$ for all sufficient large $i$ by considering $\frac{m_i}{n_i}(P_{F_i}+\frac{1}{r}T_{F_i})\simQ M_{F_i}$, but coefficients is tending to $\infty$.  This means $T_{F_i}=0$ for all $i>>0$. In particular this means $h^0(-rK_{X_i}|_{F_i})=h^0(-r(K_{F_i}+\Lambda_{F_i}))=h^0(-r\Lambda_{F_i})\neq 0$, as $\Lambda_{F_i}\leq 0$. Now we can easily derive a contradiction using Proposition \ref{lif_sec}.
\end{proof}
We will omit the proof of weak BAB, i.e Theorem \ref{wea_bab}, because it is very similar to the proof given above.

\newpage%%%%%%%%%%%%%%%%

\section{Complements}
In order to prove theorem \ref{eff_bir}, we need to show we can also find the $B$ as in \ref{4.9}. Such $B$ clearly exists if we can show theorem \ref{com_gen}. Hence our goal is to give a inductive treatment of theorem \ref{com_gen}. We will apply induction and try to construct complements from fiber space or non-klt centres. However, there are cases when we can't create any of the above. Hence we need to deal with them separately. these pairs are called exceptional pairs as defined in \ref{7}, we will deal with them in the next section. 
\subsection{Main Result}
Our goal this section is to show the following. 
\begin{prop}\label{6.15}
    Assume theorem \ref{com_gen} in dimension d-1 and theorem \ref{com_rel} in dimension d and boundedness of exceptional pairs as in theorem \ref{bou_exc} in dimension d, then theorem \ref{com_gen} hold in dimension d.
\end{prop}
\begin{proof}
    By induction, we can assume theorem \ref{com_gen} hold in dimension $d-1$. Firstly by taking $\Q$ factorial dlt models and apply lemma \ref{6.1.3}, we can reduce the problem when $X'$ is $\Q$-factorial. \\
    Step 1: We reduce to the case when coefficient of $B$ belong in some finite set. Choose $\epsilon$ as in the statement of \ref{2.50}, we see that we can assume coefficient of $B$ are either equal to 1 or less than $1-\epsilon$, but coefficients of B belong to some fixed DCC set, hence we can assume coefficients of $B$ belong to some finite set $\mathfrak{R}$. \\
    Step 2: We end the proof in the exceptional case. Now assume $(X',B'+M')$ is exceptional. Then by theorem \ref{bou_exc}, such $X'$ are bounded. let $q= pI(\mathfrak{R})$, and let $L := -q(K_{X'}+B'+M')$ is an integral divisors. Since $X'$ bounded, we can find a very ample divisor $A$ on $X'$ such that $A^d$ and $A^{d-1}(-K_{X'})$ both bounded above uniformly. Hence we get $A^{d-1}L$ is bounded above as $B'+M'$ is pseudoeffective. Therefore by \ref{bnd_fam}, we see that Cartier index of $L$ is bounded, which means that there exist a fixed $n$, such that $-n(K_{X'}+B'+M')$ is nef and Cartier. Now since $X'$ is fano type, there exist $C'$ such that $(X',C')$ is klt and $-K_{X'}-C'$ is nef and big. We can use effective base point free theorem here since $-n(K_{X'}+B'+M')-K_{X'}-C'$ is nef and big. Replacing $n$, we can assume that there is a a fixed $n$ such that $|-n(K_{X'}+B'+M')|$ is base point free. Picking $G'\in |-n(K_{X'}+B'+M')|$ general, we see that $B'^+ := B'+\frac{1}{n}G'$ is a n-complement.
    Step 3. It remains to prove assuming theorem \ref{com_gen} in dimension d-1 and theorem \ref{com_rel} in dimension d, we can prove theorem \ref{com_gen} in dimension d for non-exceptional pairs while assuming coefficients of $B'$ lie in some finite set. This will be our next proposition.
\end{proof}
\begin{prop}\label{6.11}
    Assuming theorem \ref{com_gen} in dimension $d-1$ and theorem \ref{com_rel} in dimension d. Then theorem \ref{com_gen} hold in dimension $d$ for $B'\in \mathfrak{R}$ and $(X',B'+M')$ non-exceptional.
\end{prop}
The remaining of this section will be devoted to proving the above proposition.
\subsection{Lifting Complements from Fibration}
Firstly we will settle the case when we can create some fibration. The idea is that when we run MMP, we either end up with Mori fibre space or minimal model. Here we deal with the fibre space case.

\begin{lemma}\label{6.5}
    Assuming theorem \ref{com_gen} for d-1 and theorem \ref{com_rel} for dimension d, then theorem \ref{com_gen} hold for $(X',B'+M')$ such that there is a contraction $f' : X'\rightarrow V'$ where $\dim V'>0$, $K_{X'}+B'+M'\simR 0/V'$ and $M'$ not big over $V'$.
\end{lemma}
\begin{proof}
    Assume $X'$ is $\Q$-factorial by taking $\Q$-factorial dlt models. Run MMP over $V'$ on $M'$ end with a minimal model (as $M'\geq 0$), replace $X'$ by it we can assume $M'$ is nef, hence semi-ample, over $V'$. Consider the non-birational contraction $f: X'\rightarrow T'/V'$ defined by $M'$ (as $M'$ not big over $V'$). Therefore by replacing $V'$ with $T'$, we can assume $M'\simQ 0/V'$. Also the MMP is $K_{X'}+B'+M'$ trivial, so all conditions are preserved.\\
    Step 1: We can't just pull back the complements on $V'$ since we lost control of $M'$. Firstly apply Proposition \ref{6.3} and adjunction on fiber space, we get $q(K_{X'}+B')\sim q(K_{V'}+B_{V'}+P_{V'})$ where $P_{V'}$ is the moduli part, coefficients of $B_{V'}$ belong to some $\Phi(\mathfrak{S})$ for some finite set of rationals $\mathfrak{S}$, where $q,\mathfrak{S}$ only depend on $p,\mathfrak{R}$. Furthermore if $X\xrightarrow{\phi} X'$ and $V\xrightarrow{\psi} V'$ are high resolution, we can assume $qP_V$ is nef Cartier, $f:X\rightarrow V$ is a morphism and $(V',B_{V'}+M_{V'})$ is a generalised pair with data $V\rightarrow V'$ and $P_V$. Now since $M'\simQ 0/V'$, and let $\phi^*(M')= M+E$, where $E\geq 0$. Apply (\cite{fa},2.44), we see that $qM\sim 0/V$, hence we can let $qM\sim qf^*(M_V)$ for some $\Q$-divisor on $M_T$ such that $qM_V$ is nef Cartier. Since $E$ is vertical over $V$ and $\simQ 0/V$, we have $E =f^*(E_V)$, where $E_V$ is effective. Define $M_{V'} := \psi_*(M_V)$. Since $M_V+E_V\simQ 0/T'$ and $E_V$ is exceptional over $V'$ (since $E$ is exceptional over $X'$). We see that $\psi^*(M_{V'})=M_V$, hence $qM'\sim qf'^*(M_{V'})$ by diagram chasing. \\
    Step 2: Letting $K_X+B=\phi^*(K_{X'}+B')$, and get fibre space adjunction $q(K_X+B )\sim qf^*(K_V+B_V+P_V)=qf^*\psi^*(K_{V'}+B_{V'}+P_{V'})$. (note coefficients of $B_V$ $\leq$ 1 since $(X,B)$-lc. Hence we get $\psi^*(K_{V'}+B_{V'}+P_{V'}+M_{V'})=K_V+B_V+E_V+P_V+M_V$. Therefore, $(V',B_{V'}+{P_V'}+M_{V'})$ is generalised lc pair with data $P_V+M_V$ and $q(P_V+M_V)$ is nef Cartier: Indeed, we can assume $X,V$ are sufficient high resolution such that $(V,B_V+E_V)$ is log smooth, hence it suffices to show $B_V+E_V\leq 1$. But $B+E\leq 1$ by $(X',B'+M')$ lc,$E=f^* E_V$, therefore $B_V+E_V\leq 1$ is clear from definition of $B_V$.\\
    Step 3: Hence by induction, $K_{V'}+B_{V'}+P_{V'}+M_{V'}$ has a n-complement, say $B_{V'}^+ := G_{V'}+B_{V'}$, $G_{V'}\geq 0$. Define $0\leq G' := f^{'*} G_{V'}$, $G_V := \psi^*(G)$ and $G := f^* G_V =\phi^*(G')$, which is vertical over $V$. Let $B'^+:= B'+G'$, we see that $n(K_{X'}+B'^+ +M')\sim nf'^*(K_{V'}+B_{V'}^+P_{V'}+M_{V'})\sim 0$. Also observe that $\phi^*(K_{X'}+B'+G'+M')= K_X+B+G+E+M$ and $ f^*(K_V+B_V+G_V+E_V+P_V+M_V) = f^*(\psi^*(K_{V'}+B_{V'}+G_{V'}+P_{V'}+M_{V'}))$. We can see that coefficients of $B+G+E$ are $\leq 1$: Indeed, we know  $B_V+G_V+E_V\leq 1$ (as $(V',B_{V'}^+ + P_{V'}+M_{V'})$ generalised lc), and $B+E\leq 1$ and $G = f^*(G_V)$. If there is a component $D$ of $B+E+G$ with coefficient $> 1$, then it must be a component of $G$. Hence say its image on $V$ is a divisor, $C$. Then since $\mu_C(B_V+E_V+G_V)\leq 1$, we see that $\mu_C(E_V+G_V)\leq t_C$, where $t_C$ is the lct of $f^*C$ w.r.t. $(X,B)$ over the generic point of $C$. Therefore $(X,B+E+G)$ is sub-lc over $\eta_C$, the generic point of $C$ (since $E+G := f^*(E_V+G_V)$ and over $\eta_C$, $E+G\leq t_C f^* C$), which is a contradiction. Therefore $B'^+$ is indeed an n-complement, hence we are done.
\end{proof}
\subsection{Lifting Complements from Divisorial adjunction with Plt Non-klt Centres}
Here we show that we can create lift complements if there are some plt non-klt centre around.
\begin{prop}\label{6.7}
    Assuming theorem \ref{com_gen} for d-1 and theorem \ref{com_rel} for dimension d, then theorem \ref{com_gen} hold for $(X',B'+M')$, such that $B'\in \mathfrak{R}$, $\exists (X',\Gamma'+\alpha M') \Q$-factorial plt, $\alpha\in(0,1)$, $-(K_{X'}+\Gamma'+\alpha M')$ ample and $S' =\floor{\Gamma'}$ with $S'\in \floor{B'}$. 
\end{prop}
\begin{proof}
    We will see why these assumption are clearly needed for the proof. Since $(X',\Gamma'+\alpha M')$ is plt and $-(K_{X'}+\Gamma'+\alpha M')$, we see that $(S',\Sigma_{S'}+\alpha M_{S'})$, where $K_{S'}+\Gamma_{S'}+\alpha M_{S'} := (K_{X'}+\Gamma'+\alpha M)|_{S'}$ is divisorial adjunction, is generalised klt and $K_{S'}+\Gamma_{S'}+\alpha M_{S'}$ is anti ample, hence we have $S'$ is Fano type by \cite{fa},2.13(5). Also by replacing $\Gamma'$ by $\delta\Gamma'+(1-\delta)B'$ and $\alpha$ by $\delta \alpha+(1-\delta)$ for $0<\delta<<1$, we can assume $\Gamma'-B'$ has sufficiently small coefficients. (note all assumptions are preserved since $-(K_{X'}+B'+M')$ is nef). Now letting $q(K_{S'}+B_{S'}+M_{S'})\sim q(K_{X'}+B'+M')|_{S'}$ as in \ref{div_adj}, we can assume $q$ is fixed depending only on $d,n,p$ such that $qM,qM_S$ both nef Cartier, $qB'$ is integral. Also by induction (since $(S',B_{S'}+M_{S'})$ is generalised lc as $(X',S')$ plt), there is a n-complement $B_{S'}^+:= B_{S'}+R_{S'}$ for $(S', B_{S'}+M_{S'})$ with $R_{S'}\geq 0$.\\
    Step 1: We will try to apply kawamata Viehweg vanishing to lift complements. Consider $X\xrightarrow{\phi} X $ log resolution of $(X',B'+\Gamma')$, assume $S\xrightarrow{\psi} S'$ is a morphism, where $S:= S'^\sim \in X$. Define $A := -(K_X+\Gamma+\alpha M) := -\phi^*(K_{X'}+\Gamma'+\alpha M')$ is nef and big, $N := -(K_X+B+M):=-\phi^*(K_{X'}+B'+M')$ is nef, $L:=-qK_X-qT-\floor{(q+1)\Delta}-qM =qN+q\Delta-\floor{(q+1)\Delta}$ is an integral divisor, where $T := \floor{B^{\geq 0}}$ and $\Delta := B-T$. We see that $T,\Delta$ has no common components and $S\in T$, and since $qB'$ is integral, any component of $q\Delta-\floor{(q+1)\Delta}$ is exceptional/$X'$. Now define $P\geq 0$ unique integral divisor such that $(X,\Lambda:= \Gamma+q\Delta-\floor{(q+1)\Delta}+P)$ is plt and $\floor{\Lambda}=S$:  Indeed, we can choose such $P\geq 0$. Since $\Gamma+q\Delta-\floor{(q+1)\Delta}=\Gamma-B+T+\{(q+1)\Delta\}$,if $D$ is component of $T$, then $\mu_D(P)=0$ (in particular $\mu_S P=0$), and if $D$ is not a component of $T$, we see that $0\leq \mu_D P\leq 1$ since $\Gamma-B$ sufficiently small. Hence we conclude that $P\leq 1$. Also it is clear that $P$ is exceptional over $X'$, since all component of $P$ must be
    a component of $q\Delta-\floor{(q+1)\Delta}$, which is exceptional/$X'$.\\
    Step 2: By easy computation, we see that $L+P = K_X+\Lambda+A+\alpha M+qN$ where $(A+\alpha M+qN)$ is nef and big and therefore by Kawamata Viehweg vanishing $h^1(K_X+\Lambda-S+A+\alpha M+qN)=0$, which means $$h^0(L+P)\twoheadrightarrow h^0((L+P)|_S)$$ Now let $R_S := \psi^*(R_S')\geq 0$, we see that $-qG_S\sim q(K_S+B_S+M_S)\sim -qN|_S$. Therefore $(L+P)|_S \sim G_S := qR_S+q\Delta_S-\floor{(q+1)\Delta_S}+P_S$, where $\Delta_S := \Delta|_S,P_S := P|_S$. We see that $G_S\geq 0$: indeed $G_S$ is an integer divisor and $\mu_D(G_S)\geq \mu_D(q\Delta_S-\floor{(q+1)\Delta_S})>-1$. Therefore there is $0\leq G\sim L+P$ such that $G|_S =G_S$. Let $G' := \phi_* G$, and $B'^+ := B'+R'$, where $R' := \frac{1}{n}G'$, Hence $G'\sim \phi_*(L+P)=-n(K_{X'}+B'+M')$ (since $P$ and $q\Delta-\floor{(q+1)\Delta}$ exceptional over $X'$), which shows that $n(K_{X'}+B'+M')\sim 0$.\\
    Step 3: Finally we only need to show $(X',B'^+ +M')$ is generalised lc. Firstly we are done if we show $R'|_{S'}=R_{S'}$: Indeed, if so then we have $(K_{X'}+B'^+ +M')|_{S'}=K_{S'}+B_{S'}^+ +M_{S'}$, which is generalised lc, hence $(X',B'^+ +M')$ is generalised lc near $S'$. If $(X',B'^+ +M')$ is not generalised lc, then so  $(X',\Omega'+F')$ is not generalised lc but generalised lc near $S'$, where $\Omega' := \epsilon(B'^+)+(1-\epsilon)\Gamma'$ and $F' := \epsilon M'+(1-\epsilon)\alpha M$. Furthermore $-(K_{X'}+\Omega'+F')$ is ample, this contradicts the connectedness principle (\cite{fa},2.14). Now we show $R'|_S' =R_{S'}$. Define $qR := G-q\Delta+\floor{(q+1)\Delta}-P$. It is clear that $R|_S =R_S$. Also $qR\sim L-q\Delta+\floor{(q+1)\Delta}=nN\sim 0/X'$ and $\phi_*R=R'$, we get $\phi^*{R'}:= R$ and hence $R'|_{S'} =\psi^*{R|_S}=R_{S'}$ as desired.
\end{proof}

\subsection{Complements for Strongly Non-exceptional Pairs}
Now we prove the the case for strongly non-exceptional pairs. 
\begin{prop}\label{6.8}
        Assuming theorem \ref{com_gen} for d-1 and theorem \ref{com_rel} for dimension d, then theorem \ref{com_gen} hold for $(X',B'+M')$, such that $B'\in \mathfrak{R}$, $(X',B'+M')$ not generalised klt and either $K_{X'}+B'+M'\not\simQ 0$ or $M'\not\simQ 0$.
\end{prop}
\begin{rmk}\label{6.9}We firstly remark that this implies the inductive statement hold stronly non-exceptional pairs. Indeed if stronly non-exceptional, then $\exists \Q$-divisor $ 0\leq P'\simQ -(K_{X'}+B'+M')$ such that $(X',B'+P'+M')$ is not generalised lc, in particular $P'\neq 0$. let $t<1$ be the generalised lct of $P'$ wrt $(X',B'+M')$ and $\Omega' := B'+tP'$. Let $(X'',\Omega''+M'')$ be $\Q$-factorial dlt model of $(X',\Omega'+M')$, and choose $B'^\sim \leq \theta''\leq \Omega''$ with $\theta''\in \mathfrak{R}$ and $\floor{\theta''}\neq 0$. Therefore it its enough to find n-complement for $(X'',\theta''+M'')$. Now running MMP on $-(K_{X''}+\theta''+M'')$ ends with minimal model $(X''',\theta'''+M''')$ such that $-(K_{X'''}+\theta'''+M''')$ is nef (since $-(K_{X''}+\theta''+M'')\simQ -(K_{X''}+\Omega''+M'')+(\Omega''-\theta'')\simQ (1-t)P''+(\Omega''-\theta'')$ where $P'' := P' |_{X''}$ is pseudoeffective). Finally let $P'''$ be pushdown of $P''$. Then by \ref{6.1.3}, it suffices to find n-complement for $(X''',\theta'''+M''')$. But this follows from \ref{6.8}: Indeed, its easy to see that $(X''',\theta'''+M''')$ is generalised lc but not generalised klt, and $P'''\neq 0$ since $P''\neq 0$ is nef, which means $-(K_{X'''}+\theta'''+M''')\simQ (1-t)P'''+\Omega'''-\theta'''\not\simQ 0$. Now we turn to proof of \ref{6.8}.
\end{rmk}
\begin{proof}[Proof for \ref{6.8}]
    Step 0: We firstly say why we need $K_{X'}+B'+M'\not\simQ 0$ or $M'\not\simQ 0$. First take $\Q$-factorial dlt model, we can assume $X'$ is $\Q$-factorial and $(X',B')$ dlt but not klt. Since $-(K_{X'}+B'+M')$ is nef hence semi-ample, let $X'\rightarrow V'$ be the contraction defined by it. Now run MMP on $M'$ over $V'$, we can replace $X'$ by the minimal model and assume $M'$ semi-ample /$V'$ defining contraction $X'\rightarrow Z'/V'$(this is a $-(K_{X'}+B'+M')$ trivial MMP on $(-K_{X'}-B')$,all assumption are preserved). Now if $\dim V'>0$, then assume $M'$ is big over $V'$, else we are done by \ref{6.5}. If $\dim V' =0$, then $K_{X'}+B'+M'\simQ 0$, hence $M'\not\simQ 0$, therefore $\dim Z'>0$ and we can assume $M'$ big over $Z'$. Either way we can reduce to the case that $M'$ is nef and big over $Z'$. We introduce $\alpha$ and $\beta$. for any rational $\alpha<1$ sufficiently close to 1, we have $-(K_{X'}+B'+\alpha M')=-(K_{X'}+B'+M')+(1-\alpha)M')$ is globally nef and big. The contraction defined by $-(K_{X'}+B'+\alpha M')$ is the same as the contraction defined by $M'$, call it $X'\rightarrow Z'$. Run MMP on $B'$ over $Z'$ sand replacing $X'$ we can assume $B'$ is nef over $Z'$, hence for all sufficiently close to 1 rational $\beta\in(\alpha,1)$, we have $-(K_{X'}+\beta B'+\alpha M')$ is globally nef and big. (note all these MMP preserve all of our assumptions and $(X',B')$ not klt.) \\
    Step 1: Now letting $(X'',B'')$ be $\Q$-factorial dlt model of $(X',B')$, Then it is easy to see that $(X'',B''+M'')$ is generalised dlt model for $(X',B'+M')$. Let $K_{X''}+\bar{B''}=(K_{X'}+\beta B')|_{X''}$, and replace $X'$ by $X''$. We get that $(X',B'+M')$ generalised dlt but not klt, and there exist $\alpha$ (we fix it here) close to 1 such that $-(K_{X'}+B'+\alpha M')$ nef and big and $\exists \bar{B'}$ arbitrarly close to $B'$ such that $-(K_{X'}+\bar{B'}+\alpha M')$ is nef and big and $(X',\bar{B'}+\alpha M')$ is generalised klt. Now assume $-(K_{X'}+B'+\alpha M')\simQ A'+G'$ where $G'\geq 0$ and $A'$ is ample.\\
    Step 2: We settle the case $supp G'$ doesn't contain any non-klt centre of $(X',B')$. If so, for $\delta>0$ sufficiently small, we can write $-(K_{X'}+B'+\alpha M'+\delta G')\simQ (1-\delta)(\frac{\delta}{1-\delta}A'+A'+G')$ is ample and all non-klt locus is of $(X',B'+\delta G'+\alpha M')$ is the exactly non-klt locus of $(X',B')$ and in particular $(X',B'+\delta G'+\alpha M')$ is generalised dlt. Therefore, if we let $\Gamma' := aS'+(1-a)(B'-S'+\delta G') $ for some $S'\in \floor{B'}$ and $0<a<<1$, then $(X',\Gamma'+\alpha M')$ satisfies the condition in \ref{6.7}. Now we focus on $Supp G'$ contains some non-klt centre of $(X',B')$.\\
    Step 3: We still like to add some $G'$ to $B'$, if we can't then we try to add to $\bar{B'}$. Let $t>0$ be the generalised lct of $(G'+B'-\bar{B'})$ wrt to $(X',\bar{B'}+\alpha M')$, we see that since $\bar{B'}$ sufficiently close to $B'$ and $Supp G'$ contain some non-klt centre of $(X',B')$, we see that $t>0$ but arbitrarily small depending on $\beta$. Let $\Omega' := (1-t)\bar{B'}+t B'+ tG'\geq \bar{B'}$. Again it is clear that $-(K_{X'}+\Omega'+\alpha M')$ is ample since we added a little bit of $G'$ to it and $A'+G'$ is nef and big. Also since $\bar{B'}$ arbitarily close to $B'$ and $t$ sufficiently small, we get $\floor{\Omega'}\leq \floor{B'}$. If $\floor{\Omega'}\neq 0$, then we are done by appling end of step 2. Hence we assume $\floor{\Omega'} =0$.\\
    Step 3: Finally we finish the proof in the case of $\floor{\Omega'}=0$.  Consider $(X'',\Omega''+\alpha M'')$ the $\Q$-factorial dlt model of $(X',\Omega'+\alpha M')$. Let $B'' := B'^\sim +exc/X'$, again we see that $\floor{\Omega''}\leq \floor{B''}$. Now we consider the construction of small $\Q$-factorial model as in \ref{Q_fac}. Hence we let $\Delta'' := \bar{B'}^\sim+exc/X'$ and run MMP on $(K_{X''}+\Delta''+\alpha M'')$,(essential we just keep doing extremal contraction contracting exceptional divisors from $X''\rightarrow X'$, which are components of $\floor{\Omega''}$) we know we end up with $X'$, replace $X''$ by the last divisorial contraction. Then we get $-(K_{X''}+\Delta''+\alpha M'')$ is ample over $X'$ and $-(K_{X''}+\Omega''+\alpha M'')$ is pull back of ample on $X'$. Also only exceptional prime divisor $S'$ of $X''\rightarrow X$ is a common component of $\Delta''$ and $\Omega''$ by construction. Therefore by letting $\Gamma'' := a\Delta''+(1-a)\Omega''$ for $a$ suffciently small, we see that $-(K_{X''}+\Gamma''+\alpha M'')$ ample and generalised plt. Hence we can apply \ref{6.7} to $(X'',B''+M'')$. Hence $(X'',B''+M'')$ has a n-complement, therefore so does $(X',B'+M')$.
\end{proof}
\subsection{Complements for Non-exceptional Pairs}
Finally we end this section by proving Prop \ref{6.11}
\begin{proof}[Proof of Prop \ref{6.11}]
    Since it is non-exceptional, we find $0\leq P'\simQ -(K_{X'}+B'+M')$ (note $P'$ may be 0) and we can assume $(X',B'+P'+M')$ is generalised lc but not generalised klt (if not generalised lc, then done by Remark \ref{6.9}). Now we do exactly the same as Remark, \ref{6.9}, except we don't have can't apply \ref{6.8} since it is possible $(K_{X'''}+\theta'''+M''')\simQ 0$ now. Also we can replace $(X',B'+M')$ by $(X''',\theta'''+M''')$. Hence we reduce the problem to the case when \ref{6.8} doesn't apply, i.e. when $K_{X'}+B'\simQ M'\simQ 0$. Note we also have $pM\simQ 0$ with $pM$ Cartier divisor, but $Pic(X)$ is torsion free since $X$ is fano type, hence we get $pM\sim 0$ hence $pM'\sim 0$. This means it suffices to find a bounded n such that $n(K_{X'}+B')\sim 0$. (then $B'$ itself is a n-complement). This is the following lemma.
\end{proof}
\begin{lemma}\label{6.10}
        Assuming theorem \ref{com_gen} for d-1 and theorem \ref{com_rel} for dimension d, there exist $n$ depending only on $d,\mathfrak{R}$ such that if $(X,B)$ lc of dimension d,$X$ Fano type, $K_X+B\simQ 0$ and $B\in \mathfrak{R}$, then $n(K_X+B)\sim 0$
\end{lemma}
\begin{proof}
    Firstly by applying 2.1.6 and (\cite{fa},2.48), we can reduce to the case $X'$ is $\epsilon$-lc Fano variety. The idea is we are in the situation to apply Prop \ref{4.9}, hence we will use it to construct complements.\\
    Step 1: let $n$ be in Prop \ref{4.9} and \ref{6.9}. Hence we know $|-nK_{X'}|$ defines a birational map. As in the proof of \ref{4.9}, we let $\phi^*(-nK_{X'})\sim A+R$, where $|A|$ is base point free and is the movable part and $R$ is the fixed part, and let $A',R'$ be their pushdown. By assuming $A$ is general in $|A|$, we can assume $(X',\frac{1}{n}(A'+C'))$ is lc: Indeed, if not, then $X'$ is strongly non-exceptional, hence by \ref{6.9}, it has a n-complement, $C^+$, but then we have $(X',C^+)$ lc and $C^+\in \frac{1}{n}|-nK_{X'}|$ but this gives contradiction since $\frac{1}{n}(A'+C')$ is general in $\frac{1}{n}|-nK_{X'}|$. Now let $N':= \frac{1}{2n}A'$ and $N:= \frac{1}{2n}A$ is nef. Now we consider $(X',\Delta'+N':= \frac{1}{2}B'+\frac{1}{2n}R'+\frac{1}{2n}A')$.\\
    Step 2: If $(X',\Delta'+N')$ is klt, its clear that we suffice to show there exist bounded $m$ such that $m(K_{X'}+\Delta'+N')\sim 0$. Also by (\cite{fa},2.48), we get $\exists \epsilon'$ such that $(X',\Delta'+N')$ is $\epsilon'$-lc. (Note $\epsilon'$ only depend on $n,\mathfrak{R}$). Now by \cite{ACC} ACC, Cor 1.7, $(X',\Delta'+N')$ lies in some bounded family, which implies the Cartier index of $K_{X'}+\Delta'+N'$ is bounded. But $Pic(X)$ is torsion free since $X$ Fano, so we are done.\\
    Step 3: Assume it is not klt. Then it is also not generalised klt since $\phi^*(A')\geq A$ by negativity lemma. Hence clearly $N'\not \simQ 0$ since its big. Therefore by \ref{6.8}, there exist a bounded $m$ such that $(X', \Delta'+N')$ has a $m$-complement $\Delta^+\geq \Delta'$. However we have $\Delta^+-\Delta'\simQ 0$. Hence $\Delta^+=\Delta'$ and we are done again.
\end{proof}

\newpage

%%%%%%%%%%%%%%%%%%%%%%
\section{Boundedness of Exceptional Pairs}
Our goal here is to show theorem \ref{bou_exc} using induction. We first develop a set of standard tools that are needed to for \ref{bou_exc}. The following is the main tool to show boundedness.

\begin{prop}\label{19.1.3}(\cite{BCL},19.1.3)
    Let $I\subset[0,1]$ be a DCC set and $n\in \N$. If $(X,B)$ is klt pair, $B$ big and $X$ dim n, $K_X+B\simQ 0$, and $B\in I$, then such $(X,B)$ belongs to a bounded family depending only on $I$ and $n$. In particular, if X is Fano type and has a klt n-complement for bounded n, then $X$ is bounded.    
\end{prop}
\subsection{Method from Bounds on lct to Boundedness of Varieties}
Firstly we state and prove a result that is crucial in proving boundedness of the certain class of weak Fano varieties. Also, this is very important for the proof of BAB later.

\begin{prop}\label{7.13}
    Let $d,m,v\in \N$ and $t_l$ be sequence of positive real numbers. Assuming theorem \ref{com_gen} for d-1 and theorem \ref{com_rel} for dimension d. Let $\mathcal{P}$ be the set of projective varieties $X$ such that $X$ is a klt weak Fano variety of dimension d, $K_X$ has a m-complement, $|-mK_X|$ defines a birational map, $vol(-K_X)\leq v$ and $\forall l\in \N, L\in|-lK_X|$, the pair $(X,t_l L)$ is klt. Then $\mathcal{P}$ is a bounded family.
\end{prop}
\begin{proof}
    Step 0: Note that we can take small $\Q$-factorization and assume $X$ is $\Q$-factorial. Let $0\geq M\in |-mK_X|$ be general. We use standard notation as in \ref{4.4.3} and write $M=A+R$. Hence we get a  $(\bar{X},\Sigma_{\bar{X}})\in \mathcal{Q}$, a bounded family with $\bar{X}\dashrightarrow X$ such that all of \ref{4.4.3} hold. Let $X'$ be common resolution of $\bar{X},X$ also as in \ref{4.4.3}. Now since we assume $K_X$ has a m-complement, but since $(A+R)$ is general in linear system of $|-mK_X|$, we can assume the m-complement is just $B := \frac{1}{m}(A+R)$. Hence $(X,\frac{1}{m}(A+R))$ is lc. If it is klt, we are done by \ref{19.1.3}. Hence we assume it is not klt. The goal now is to construct some other boundary to apply \ref{19.1.3}.\\
    Step 1: The idea is to create some klt complement from $B$. We observe that when crepant pulling back $K_X+B$ to $\bar{X}$, we get $\floor{B_{\bar{X}}}$ are contained inside, by construction, $\Sigma_{\bar{X}}$. Hence we can add on some $A_{\bar{X}}$ and remove some $\Sigma_{\bar{X}}$ and make it sub $\epsilon$ lc. Finally in order to correct the negative terms we just added, we use complements.The proof is completed now so it remains to make it formal\\ Step 2: The formal argument is as below. Replacing $m$ at the start we can assume $A$ is not a component of $\floor{B}$. Define $K_{\bar{X}}+B_{\bar{X}}$ be crepant pullback of $K_X+B$. Since $supp(B_{\bar{X}}), supp(A_{\bar{X}})\subset \Sigma_{\bar{X}}$ and $A_{\bar{X}}$ is big and $(\bar{X},\Sigma_{\bar{X}})$ is bounded, we can find a bounded $l$ depending only on $\mathcal{Q}$ such that $lA_{\bar{X}}\sim G_{\bar{X}}\geq 0$ such that $\Sigma_{\bar{X}}\leq G_{\bar{X}}$. Since $A_{X'}\sim 0/\bar{X}$, we get $lA'\sim G':= G_{\bar{X}}|_{X'}$ and hence $lA\sim G$ the pushdown to $X$. Now we make sure we can correct the negative terms of subtracting $G$ using complements. Notice $G+lR\in |-lmK_X|$, hence by assumption $(X,t(G+lR))$ is klt when $t := t_{lm}$, we can wlog $t<\frac{1}{lm}$. Now if $(X,\frac{1}{lm}(G+lR))$ is lc, then $\Omega := \frac{1}{lm}(G+lR)$ is a $lm$ complement, else we have $(X,\frac{1}{lm}(G+lR))$ not lc, hence $(X,t(G+lR))$ is strongly non-exceptional, therefore there exist a $n(d,t)$-complement $\Omega \geq t(G+lR)$. Now consider $\Theta := B+\frac{t}{m}A-\frac{t}{ml}G$ for $t$ sufficiently small. Since $K_X+\Theta\simQ 0$, we can see it is sub $\epsilon(t,l,m)$-lc since its crepant pullback is just $B_{\bar{X}}+\frac{t}{m}A_{\bar{X}}-\frac{t}{ml}G_{\bar{X}}$. 
    Now let $\Delta :=\frac{1}{2}(\Omega+\Theta)$, we see $K_X+\Delta\simQ 0$ and $(X,\Delta)$ is $\epsilon/2$-lc (since $\Delta\geq 0$ by construction). Also we see that coefficients of $\Delta$ belong to a finite set depending only on $t,m,l,n$. Hence we can apply \ref{19.1.3} again and we are done.
\end{proof}
Hence the remaining of the section is basically verify that all the criteria in \ref{7.13} is satisfied.

\subsection{Various Bounds on Exceptional Pairs}
Here we discuss various bounds that exist on exceptional pairs.
\begin{lemma}[Bounds on Singularities]\label{7.2}
    Let $d,p\in \N$, and $\Phi\subset [0,1]$ a fixed DCC set. There exists $\epsilon>0$ such that if $(X',B'+M')$ exceptional pair, $X$ fano type, $B'\in \Phi$ and $pM$ b-Cartier, then $\forall 0\leq P'\simR -(K_{X'}+B'+M')$, $(X',B'+P'+M')$ is $\epsilon$-lc.
\end{lemma}
\begin{proof}
    We proceed by contradiction. Assume $\epsilon$ is arbitrarily close to 0. Firstly note $(X',B'+P'+M')$ is generalised klt. pick $D$, prime divisor on birational model of $X$ such that $a:=a(D,X',B'+P'+M')$ is minimal. Consider extraction of divisor $D$ via $X''\xrightarrow{\phi} X'$, (see \ref{Q_fac}). Write $K_{X''}+B''+P''+M''=\phi^*(K_{X'}+B'+P'+M')$, where $P'':=\phi^* P'$. Then say $\mu_D(B'')=c$ and $\mu_D(P'')=e$, we have $c+e=1-a$ by construction. We want to apply 2.50, hence we run MMP on $-(K_{X''}+B''+cD+M'')\simR P''-cD\geq 0$, we end with a minimal model say $(X''',B'''+cD'''+M''')$. Define $\theta''' := (B'''+cD''')^{\leq 1-\epsilon}+\ceiling{(B'''+CD''')^{>1-\epsilon}}\geq B'''+cD'''$, and we have $\mu_{D'''}\theta'''=1$, where $\epsilon$ is given in 2.50 (we assume $a< \epsilon$). Finally run MMP on $-(K_{X'''}-\theta'''+M''')$ we end with minimal model $(\bar{X},\bar{\theta}+\bar{M})$, where $\mu_{\bar{D}}\bar{\theta} =1$ by \ref{2.50}. However we can let $X$ be the common resolution of $X',X'',X''',\bar{X}$. Then we have $(K_{\bar{X}}+\bar{\theta}+\bar{M})|_X\geq (K_{X'''}+\theta'''+M''')|_X \geq (K_{X''}+B''+cD''+M'')|_X \geq (K_{X'}+B'+M')|_X $, which implies $(\bar{X},\bar{\theta}+\bar{M})$ is exceptional by \ref{7}, this is contradiction as $(\bar{X},\bar{\theta}+\bar{M})$ is not generalised klt and $-(K_{\bar{X}}+\bar{\theta}+\bar{M})$ is semi-ample, (which means we can define exceptional).
    \end{proof}
\begin{lemma}\label{7.7}[Bound on exceptional threshold]
    Let $d,p\in \N$, and $\Phi\subset [0,1]$ a fixed DCC set. There exists $\beta\in(0,1)$ such that if $(X',B'+M')$ exceptional pair of dimension d, $X$ fano type and $\Q$-factorial, $B'\in \Phi$,$pM$ b-Cartier and $-(K_{X'}+B'+M')$ is nef, then $(X',B'+\alpha M')$ is exceptional for all $\alpha\in(\beta,1)$
\end{lemma}
\begin{proof}
    Assume the lemma is false, then there is $(X_i',B_i'+M_i')$ and $\alpha_i$ such that $(X_i',B_i'+\alpha_i M_i')$ not exceptional and $\alpha_i\rightarrow 1$. We firstly note that it does make sense to say $(X_i',B_i'+\alpha_i M_i')$ is not exceptional since $-(K_{X_i'}+B_i'+\alpha_i M_i')$ is nef. Hence there exist $0\leq P_i' \simR -(K_{X_i'}+B_i'+M_i')$ such that $(X',P_i'+B_i'+\alpha_i M_i')$ is not generalised klt. then we apply the standard method of letting $\Omega_i' := B_i'+t_i P_i'$ where $t_i$ is the lct of $P_i'$. Consider $\Q$-factorial dlt model $(X'',\Omega''+\alpha_i M'')$ of $(X',\Omega'+\alpha_i M')$. Now choose $B_i^\sim\leq \Gamma_i''\leq \Omega_i''$ such that $\floor{\Gamma_i''}\neq 0$ and $\Gamma_i''\in \Phi$. The idea is to apply ACC theorem in some form to derive contradiction. We observe that $-(K_{X_i''}+\Gamma_i''+\alpha_i M_i'')=(1-t_i)P_i''+(\Omega_i''-\Gamma_i'')\geq 0$ where $P_i''$ is pull back of $P_i'$. Now we run MMP on $-(K_{X_i''}+\Gamma_i''+M_i'')$ and we end with $(\bar{X_i},\bar{\Gamma_i}+\bar{M_i})$. If we can show that $(\bar{X_i},\bar{\Gamma_i}+\bar{M_i})$ is a lc minimal model then we are done, since it is not exceptional by construction, hence we get $(X',B'+M')$ not exceptional, which is a contradiction.\\
    Now we show $(\bar{X_i},\bar{\Gamma_i}+\bar{M_i})$ is a lc minimal model. Notice that $-(K_{X_i''}+\Omega_i''+\alpha_i M_i'')$ is nef by construction and $(X'',\Omega_i''+\alpha_i M_i'')$ is generalised lc, we get by negativity $(\bar{X},\bar{\Omega_i}+\alpha_i \bar{M_i})$ is generalised lc. Now by ACC (\cite{GACC},[9],1.5) and $\alpha_i\rightarrow 1$, we must have $(\bar{X},\bar{\Omega_i}+\bar{M_i})$ is generalised lc for all $i>>0$. Now assume the MMP end with Mori Fibre space, we get $\bar{X_i}\rightarrow T_i$ extremal contraction with $\dim T_i <d$ and we have $(K_{\bar{X_i}}+\bar{\Gamma_i}+\bar{M_i})$ ample over $T_i$. If we let $\lambda_i$ be the maximum such that $(K_{\bar{X_i}}+\bar{\Gamma_i}+\lambda_i\bar{M_i})$ is nef over $T_i$, we get $K_{\bar{X_i}}+\bar{\Gamma_i}+\lambda_i\bar{M_i}\equiv 0/T_i$. Then by apply ACC on general fibres we see that $\lambda_i$ is bounded away from 1. Hence we get $K_{\bar{X_i}}+\bar{\Gamma_i}+\alpha_i\bar{M_i}$ is ample over $T_i$. This is a contradiction since $K_{\bar{X_i}}+\bar{\Gamma_i}+\alpha_i\bar{M_i}\simR -((1-t_i)\bar{P_i}+(\bar{\Omega_i}-\bar{\Gamma_i})\leq 0$ which is a contradiction.
\end{proof}
We need this in order to get a bounded $m$ such that $|-mK_X|$ defines a birational map using \ref{4.9}.
\begin{prop}\label{7.9}[Bounds on Volume]
    Let $d,p\in \N$ and $\Phi\subset[0,1]$ a fixed DCC set. Then there is $v(d,p)>0$ such that if $(X',B'+M')$ generalised klt of dimension d, $B'\in \Phi$, $pM$ is b-Cartier and big, and $K_{X'}+B'+M'\simR 0$, then $vol(-K_{X'})< v$
\end{prop}
\begin{proof}
    We firstly note that it is clear that $X'$ is Fano type since $M'$ is big. Also it is clear that it is exceptional. By taking small $\Q$-factorization, we can assume $X$ is $\Q$-factorial. (note that $\Q$-factorization preserves $K_{X'}+B'+M'\simR 0$, but may not preserve exceptionalness). Now assume $vol(-K_{X_i'})\rightarrow \infty$. By boundedness of extremal ray, we get $K_{X_i}+3dpM_i$ is big, hence $vol(-K_{X_i'})\leq vol(3dpM_i')$, which means $vol(M_i)\rightarrow \infty$.  Then there exist $\delta_i\rightarrow 0$ rationals such that $vol(-\delta_i M_i)>2d^d$. Hence we get $vol(-(K_{X_i'}+B'+(1-\delta_i)M_i))>(2d)^d)$, hence for general $x,y$ there exist $0\leq \Delta_i \simR -(K_{X_i'}+B'+(1-\delta_i)M_i')$ such that $(X_i',\Delta_i)$ is klt but not lc at $x$ and not lc at $y$. In particular, $(X_i',B_i+(1-\delta_i)M_i')$ is not exceptional. However, $1-\delta_i\rightarrow 1$, this contradicts \ref{7.7}.
\end{proof}
Now we will prove the following lemma which is a requirement for \ref{7.13}
\begin{lemma}\label{7.11}[Boundedness on lc threshold]
    Let $d,p,l\in \N$ and $\Phi\subset [0,1]$ a fixed DCC set. There exist $t(d,p,l,\Phi)>0$ such that if $(X',B'+M')$ exceptional of dimension d, $B\in \Phi$, pM big and b-Cartier, $-(K_{X'}+B'+M')$ nef and $X'$ fano type and $\Q$-factorial. then $\forall L'\in |-lK_{X'}|$, $(X',tL')$ is klt.
\end{lemma}
\begin{proof}
    Assume not, then there exist $t_i\rightarrow 0$ and $(X_i',B_i'+M_i'),L_i'$ as in the lemma with $(X,t_iL_i')$ not klt. By \ref{7.7}, there is a $\beta<1$ rational, such that $(X_i',B_i'+\beta M_i')$ exceptional. Let $s_i := lct(L_i',X_i',B_i'+\beta M_i')$, by assumption, $s_i\rightarrow 0$. Also again, we have $K_{X_i'}+3dp M_i'$ is big, hence $-\frac{1}{l}L_i' +3dp M_i$ is big. Now we get $-(K_{X_i'}+B_i'+s_i L_i+\beta M_i)$ is big by above. This means we can find $-(K_{X_i'}+B_i'+s_i L_i+(1-\beta)M_i)\simR P_i\geq 0$. This contradicts the exceptionalness of $(X_i',B_i'+\beta M_i')$.
\end{proof}
\subsection{Proof of Theorem \ref{bou_exc}}
Firstly we need to show that the set of exceptional weak fano variety is bounded.
\begin{lemma}\label{7.5}
	Assuming theorem \ref{com_gen} for d-1 and theorem \ref{com_rel} for dimension d, then the set of exceptional weak fano variety form a bounded family.
\end{lemma}
\begin{proof}
	Step 0: Assume the lemma is false, i.e. there is $X_i$ are exceptional weak fano that no subsequence is bounded. By standard procedure,as in 2.1.6, we can reduce to the case when $X_i$ are fano varieties. Let $\epsilon_i$ to be the minimal log discrepancy of $X_i$ and let $\epsilon=limsup \epsilon_i$. If $\epsilon=1$, then by \ref{4.11} and potential passing to a subsequence, there exist a fixed $m\in \N$ such that $|-mK_{X_i}|$ defines a birational map. Hence picking $0\leq C_i \in |-mK_{X_i}|$, we have $(X_i,\frac{1}{m}C_i)$ klt as $X_i$ exceptional. Hence applying \ref{19.1.3}, we get $X_i$ bounded. Hence it suffice to prove the case when $\epsilon<1$.\\
	Step 1 (Direct MMP): Note by using $X_i$ exceptional, and \ref{19.1.3}, it suffices to show that $X_i$ has a n-complement for some bounded $n$. We wlog $\epsilon_i\rightarrow \epsilon$ by passing to subsequence. Extract $D_i$ via $X_i'\xrightarrow{\phi_i} X_i$ with $a(D_i,X_i,0)=\epsilon_i$, i.e. we have $K_{X_i'}+(1-\epsilon_i)D_i =\phi_i^*(K_{X_i})$. Now we claim that there is a $-D_i$ MMP and which ends in $X_i''\rightarrow T_i$ Mori fiber space such that  there is some $t_i\geq 0$ with $K_{X_i''}+(1-\epsilon_i)D_i''+t_i D_i''$ globally anti-nef and $\equiv 0/T_i$ and all extremal ray in the MMP intersects $K_{X_i''}+(1-\epsilon_i)D_i''+t_i D_i''$ non-negatively: Indeed let $s_i$ be maximum such that $-(K_{X_i'}+(1-\epsilon_i)D_i'+s_i D_i')$ is nef) and let $R$ be the extremal ray corresponding to it. hence we get from construction $(K_{X_i'}+(1-\epsilon_i)D_i'+s D_i')R=0$ and $D_i\cdot R<0$, which means it is a $_-D_i$ MMP. Now we contract $R$. We stop if we hit a MFS and set $t_i=s_i$, else we continue decreasing $s$ as above. Also note cleary from construction, we have $-(K_{X_i'}+(1-\epsilon_i)D_i'+t_i D_i')R\geq -(s_i-t)D_i' \cdot R\geq 0$. Also note by negativity we have $(X_i'', (1-\epsilon_i)D_i''+t_i D_i'')$ is also exceptional.\\
	Step 2: Define $e_i := (1-\epsilon_i)+t_i \geq (1-\epsilon_i)$. Hence we have $-(K_{X_i''}+e_i D_i'')$ is semi-ample, so there exist $0\leq P_i''\simQ -(K_{X_i''}+e_i D_i'')$. Let $K_{X_i}+P_i$ be the crepant pullback of $K_{X_i''}+e_i D_i''+P_i''$. Notice we have $P_i\geq 0$ by negativity and all extremal rays in the MMP are $K_{X_i''}+e_i D_i''$ non-negative. Now by \ref{7.2}, we have $\exists \delta>0$ independent of $i$ such that $(X_i,P_i)$ is $\delta$-lc. Hence $(X_i'',e_i D_i'')$ also $\delta$-lc. Now apply \ref{wea_bab} to the general fibers of $X_i''\rightarrow T_i$ (since we have $K_{X_i''}+e_i D_i''\simQ 0$ on general fibres and $e_i$ is bounded away from 0), which, in term by \cite{fa} 2.22, implies $e_i$ lies in some finite set. Now if $\dim T_i>0$, then by \ref{6.5}, $K_{X_i''}+e_i D_i''$ has a n-complement, say $B_i$. If $\dim T_i=0$, then $X_i''$ now form a bounded family, hence the Cartier index of $K_{X_i''}+e_i D_i''$ is bounded, which implies it also has a n-complement. Either way, by pulling complements back, we see that $X_i$ also has a n-complements $B_i$, say. This implies $(X_i,B_i)$ is klt as $X_i$, exceptional. Hence again we have $X_i$ bounded which is a contradiction.
\end{proof}
\begin{proof}[Proof of Theorem \ref{bou_exc}]
	The idea is the same as the proof of \ref{7.5}. We first construct compliments and we are done if it is klt. If it is not, we will try to modify the complement so eventually we get klt n-complements and we can apply \ref{19.1.3}. \\
	Step 0: Firstly note it is clear it suffices to prove $X'$ form a bounded family and apply \ref{bnd_fam}. Hence it suffices to show $X'$ has a klt a-complement for some bounded a. Firstly we show we can construct some lc m-complements. Consider applying 2.1.6, and applying \ref{4.9}, we can assume $|-mK_{X'}|$ (m bounded) defines a birational map and take the m-complement, $C':= \frac{1}{m}(A'+R')$ using the same notation as in the proof of \ref{4.9}. Hence we assume $(X',C')$ is lc but not klt. \\
	Step 1: We construct $\Delta':= \frac{B'}{2}+\frac{1}{2m}R'$ and $N := \frac{1}{2}M+\frac{1}{2m}A$ is nef. Hence we see that $(X',\Delta'+N')$ is generalised lc and $-(K_X'+\Delta'+N')$ is nef. If $(X',\Delta'+N')$ is not exceptional, then we can find a $l$-complement $\Delta'^+ := \Delta'+G'$(l bounded). By letting $B'^+ := B'+2G'$, we can show that $(X',B'^+ +M')$ is klt and exceptional and $(X'+\frac{1}{2}(B'^++\Delta'^+)+\frac{1}{2}(M'+N')$ is generalised klt and exceptional. Hence we can assume $K_{X'}+B'+M'\simQ 0$ and generalised klt and exceptional. Alternative method can be used to make sure $K_{X'}+B'+M'\simQ 0$ and generalised klt and exceptional when $(X',\Delta'+N')$ is exceptional. \\
	Step 2:  we are ready to apply \ref{7.13}. Now since $K_{X'}+B'+M'\simQ 0$, we can replace $X'$ by $\bar{X'}$ (we obtain $\bar{X'}$ by MMP on $-K_{X'}$ and note all assumption are preserved.) Hence we get there exist bounded $m$ such that $|-mK_{X'}|$ and $K_{X'}$ has a m-complement. Also we can apply \ref{7.9} and \ref{7.11}, and then apply \ref{7.13} to get all these $X'$ lie in a bounded family. Hence Cartier index of $K_{X'}$ is bounded, hence there exist a bounded $s$ such that $|-sK_{X'}|$ is base point free, which implies $K_{X'}$ has a klt s-complement, hence we are done.
\end{proof}
\section{Final Induction Step}
Finally we will show theorem \ref{com_rel} can be showed inductively.
\begin{proof}[Proof of : theorem \ref{com_gen} and \ref{com_rel} in dimension $d-1$ $\implies$ theorem \ref{com_rel} in dimension $d$]
	We omit the proof here since the proof are essentially the same, except some minor adjustment.
\end{proof}

\newpage
%%%%%%%%%%%%%%%%%%%%%%
\section{Extra Preliminaries for Proof of BAB}
\subsection{Sequences of Blow Up}\label{blowup}
We will introduce some notation on blowups and basic properties that will be useful for the proof of theorem \ref{1.6}. Let $X$ be a smooth variety and let $X_p\rightarrow X_{p-1}\rightarrow ...\rightarrow X_1\rightarrow X_0 := X$ be a sequence of smooth, i.e. each $X_{i+1}\rightarrow X_i$ be a blowup along a smooth subvariety $C_i$ on $X_i$ of codimension $\geq 2$. We call $p$ the length of blowup and the exceptional divisor of $X_{i+1}\rightarrow X_i$ is called $E_i$.\\
Let $\Gamma$ be a reduced divisor where $(X,\Gamma)$ is log smooth. We say such a blow up sequence as above is toroidal with respect to $(X,\Gamma)$ if for each $i$, there centre $C_i$ is a stratum of $K_{X_i}+\Gamma_i$, the pullback of $K_X+\Gamma$. Note this is equivalent to saying all $E_i$ are lc place of $(X,\Gamma)$. \\
Let $T$ be a divisor over $X$. Assume for each $i$, $C_i$ is the centre of $T$ on $X$. Then we call it a sequence of centred blowup associated to $T$. By \cite{km}, 2.45, after finitely many steps $T$ is obtained on $X_p$, i.e. $T$ is the exceptional divisor $E_p$. We say $T$ is obtained by a centred blowup of length $p$. \\
Now suppose further we have we sequence of centred blowup for $T$ toroidal with respect to $(X,\Gamma)$ of length $p$ and $E_p =T$ (i.e. $T$ is a lc place of $(X,\Gamma)$. Then we have $\mu_T \phi^*(\Gamma) \geq p+1$, where $\phi :X_p\rightarrow X$: Indeed we clearly have, by construction, $C_i\subset E_i$ for each $i$. For $0\leq i< p$, we have $C_i$ is a codimension $\geq$ 2 lc centre of $(X_i,\Gamma_i)$. Also by construction, when $i>0$, we have $C_i\subset E_i$. Hence we get $\mu_{E_{i+1}}\phi_{i+1}^* \Gamma \geq \mu_{E_i}\phi_i^* \Gamma+1\geq i+2$ by induction.

\section{Preparation for the Proof of BAB}
Our goal in this section is to show theorem \ref{1.6} implies theorem \ref{1.4}. The basic idea is to run MMP and end with a Mori fiber space and either apply induction or we have to deal with the picard number 1 case separately, which is the following lemma.
\begin{lemma}\label{3.1}
	Let $d\in \N$ and $\epsilon>0$. Assuming Theorem \ref{1.6} in dimension $d$ and theorem \ref{BAB} in dimension $d-1$, there exists $v(d,\epsilon)>0$ such that if $X$ is $\epsilon$-lc Fano, $\rho(X)=1$ and $0\leq L\simR -K_X$, then $L\leq v$.
\end{lemma}
Given this we will show the main result here.
\begin{prop}
	Assuming Theorem \ref{1.6} in dimension $d$ and theorem \ref{BAB} in dimension $d-1$, then theorem \ref{1.4} hold in dimension d.
\end{prop}
\begin{proof}
	Fix $\epsilon'\in(0,\epsilon)$ and pick $0\leq L \simR -(K_X+B)$, let $s$ be largest such that $(X,B+sL)$ is $\epsilon'$-lc, hence it suffices to show $s$ is bounded away from 0. Choose $T$ such that $a(T,X,B+sL)=\epsilon'$. Replacing $X$ with small $\Q$-factorization, we assume $X$ is $\Q$-factorial. Let $\phi : Y\rightarrow X$ be the extraction of $T$ as in\ref{Q_fac}. Let $K_Y+B_Y := \phi^*(K_X+B)$ and $L_Y := \phi^* L$. By construction and $\epsilon$-lc of $(X,B)$ we get $\mu_T L_Y\geq \frac{\epsilon-\epsilon'}{s}$. Therefore it suffice to show $\mu_T L_Y$ is bounded above.\\ Now as $(Y,B_Y+sL_Y)$ is klt weak log fano, we run MMP on $-T$ and end with a Mori fiber space $Y'\rightarrow Z$(Note $T$ is not contracted). Since $-(K_Y+B_Y+sL_Y)$ is nef and big, we get $(Y',B_{Y'}+sL_{Y'})$ is also $\epsilon'$-lc. Also we have $-(K_{Y'}+B_{Y'}+sL_{Y'})\simR (1-s)L_Y'\geq 0$. If $\dim Z=0$, then $Y'$ is fano and $\rho(Y')=1$, hence by \ref{3.1}, we have $\mu_T(1-s)L_{Y'}$ is bounded above, which implies $\mu_T L_{Y'}$ is bounded away (we may assume $s\leq \frac{1}{2}$). If $\dim Z>0$, then by restricting to general fiber of $Y'\rightarrow Z$, and apply induction, we see coefficients of components of $(1-s)L_{Y'}$ horizontal over $Z$ is bounded above. In particular, again we have $\mu_T (1-s)L_{Y'}$ is bounded from above.
\end{proof}
\begin{proof}[Proof of \ref{3.1}]
	Let $T$ be a component of $L$, since $\rho(X)=1$, we have $L\equiv uT$ for some $u>0$, Hence it suffice to show $u$ is bounded from above since $\mu_T L\leq u$. we will hence assume $L =uT$, (note we only have $L\equiv -K_X$), which means $L$ is still ample.  Also we may assume $u\geq 1$ else the claim is trivial. Now by theorem \ref{com_gen}, $K_X$ has a $n(d)$-complement $K_X+\Omega$, and by theorem \ref{eff_bir}, we can also assume $|-nK_X|$ defines a birational map and $vol(-K_X)$ is bounded above.\\
	Now we apply lemma \ref{4.4.3} and we will use the same notation. Hence we get a bounded log smooth family $(\bar{X},\Sigma{\bar{X}}) \in \mathcal{P}$. If $X'\rightarrow X,\bar{X}$ be a common resolution and define $A_{X'},A_{\bar{X}},A$ similarly. Note we have $A_{X'}$ is the pullback of $A_{\bar{X}}$ and $\Sigma_{\bar{X}}$ contain the expcetional divisor of $\bar{X}\dashrightarrow X$ and birational transform of $\Omega := \frac{1}{m}(A+R)$ (since $A+R$ is a general element in $|-mK_X|$) by construction. Now let $K_{\bar{X}}+B_{\bar{X}} $ and $K_{\bar{X}}+\Omega_{\bar{X}}$ be crepant pullback of $K_X+B$ and $K_X+\Omega$ and we have $(\bar{X},B_{\bar{X}})$ is sub $\epsilon$-lc and $(\bar{X},\Omega_{\bar{X}})$ is sub-lc. Also $\Omega_{\bar{X}}\leq \Sigma_{\bar{X}}$, hence we get $a(T,\bar{X},\Sigma_{\bar{X}})\leq a(T,X,\Omega)\leq 1$.\\
	Step 2:  Now we bound certain coefficients we just defined. Let $\Omega_{\bar{X}}$ be the pushdown of $\Omega|_{X'}$ to $\bar{X}$. Since $Supp(A_{\bar{X}})\subset\Sigma_{\bar{X}}$ and $A_{\bar{X}}$ is big, there exist $l$ and very ample $H$ very ample depending only on $\mathcal{P}$, such that $lA_{\bar{X}}-H$ is big. Therefore $\bar{\Omega} H^{d-1} = \Omega|_{X'} H|_{X'}^{d-1} \leq vol(H|_{X'}+\Omega|_{X'})\leq vol(lA_{X'}+\Omega|_{X'})\leq vol(lA+\Omega)\leq vol(-(lm+1)K_X)$ is bounded above, hence $\bar{\Omega}$ is bounded from above. Hence using $(\bar{X},\Omega_{\bar{X}})$ is sub-lc, we see that if $K_{\bar{X}}+\Gamma_{\bar{X}}$ be the pullback of $K_X$  to $\bar{X}$, then negative coefficients of $\Gamma_{\bar{X}}$ is bounded from below since $\Gamma_{\bar{X}}+\bar{\Omega}=\Omega_{\bar{X}}\leq 1$. Therefore we deduce negative coefficients of $B_{\bar{X}}$ is bounded from below, hence there exists $a\in(0,1)$ depending only $\mathcal{P}$ such that $\Delta := a B_{\bar{X}}+(1-a)\Sigma_{\bar{X}}\geq 0$ and $(\bar{X},\Delta)$ is $a\epsilon$-lc. Further more by replacing $H$ by a bounded multiple we can assume $H- \Omega_{\bar{X}}$ is ample (note we have showed by $\Omega_{\bar{X}}$ has coefficients bounded from above and below). Also since $B_{\bar{X}}\simR \Omega_{\bar{X}}$, we have $H-\Delta$ is ample and we can assume there is $r$ depending only on $\mathcal{P}$ such that $H^d\leq r$.\\
	Step 3: Let $M$ be the pushdown of $T|_{X'}$ to $\bar{X}$. If $T$ is a divisor on $\bar{X}$, then we are done since $uT\equiv \Omega_{\bar{X}}$, and we get $u$ is bounded from above. Hence we assume $T$ is exceptional over $\bar{X}$. Hence support of $M$ is in $\Sigma_{\bar{X}}$, therefore we wlog $H-M$ is ample. Now applying theorem \ref{1.6}, we get there exist $t>0$ depending only on $\epsilon,d,r$ such that $(\bar{X},\Delta+tM)$ is klt. Hence $t>\frac{1}{u}$ as $a(T,\bar{X},\Delta+\frac{1}{u}M)\leq a(T,\bar{X},\Delta )-1\leq 0$ by negativity since $T$ is ample.
\end{proof}
We also need a theorem for complements in the following form which has almost the same proof as \ref{6.7}. Hence we omit the proof.
\begin{prop}\label{1.7}
	Let $d\in \N$ and $\mathfrak{R}\subset[0,1]$ be a finite set of rational numbers. There exists $n$ divisible by $I(\mathfrak{R})$ depending only on $d,\mathfrak{R}$ such that if $(X,B)$ projective lc pair of dimension $d$ with $B\in \mathfrak{R}$, $(X,0)$ is $\Q$-factorial klt, $S$ is a component of $\floor{B}$, $M$ semi-ample Cartier divisor on $X$ with $M|_S\sim 0$ and $M-(K_X+B)$ is ample, then there is a divisor $0\leq G\sim (n+1)M-n(K_X+B)$ such that $(X,B^+:= B+\frac{1}{n}G)$ is lc near $S$.
\end{prop}
Also we need a variant of \ref{4.2}.
\begin{lemma}\label{5.1}
	Let $X$ be a smooth projective variety of dimension $d$ with $A$ very ample divisor on $X$ with $A^d\leq r$. Let $L\geq0$ be a $\R$-divisor such that $deg_A L\leq r$. Then there exists $t(d,r)>0$ such that $(X,tL)$ is klt.
\end{lemma}
\begin{proof}
	The outline of the proof: cut by general hyperplane section of $|A|$ and apply induction.	
\end{proof}
\begin{prop}[\cite{fa}, Theorem 1.6]\label{volumn_bnd}
	Assume Theorem \ref{BAB} hold in dimension $d-1$, then there exists $v>0$ such that if $X$ is weak $\epsilon$-lc weak fano of dimension d then $vol(-K_X)\leq v$.
\end{prop}
\newpage

\section{Proof of BAB and Theorem \ref{1.6}}
Now we are ready to start proving theorem \ref{1.6}. Firstly it is easy to see that \ref{1.6} can be easily reduced to a statement regarding boundedness of coefficients in blowup as in the next proposition. 
\begin{prop}\label{5.7}
	Let $d,r,\epsilon$ as in theorem \ref{1.6} and let $n\in \N$. Assume theorem \ref{1.6} in dimension $d-1$, then $\exists \; q(d,r,\epsilon,n)>0$ such that if $(X,B)$ is $\epsilon$-lc, $A$ very ample with $A^d\leq r$, $\Lambda \geq 0$ with $n\Lambda$ integral, $L\geq 0$ an $\R$-divisor, $A-B,A-\Lambda,A-L$ ample, $(X,\Lambda)$ lc near $x$, $T$ is a lc place of $(X,\Lambda)$ with centre the closure of $x$, $a(T,X,B)\leq 1$, then for any $\phi: W\rightarrow X$ resolution such that $T$ is a divisor on $W$, we have $\mu_T \phi^* T\leq q$.
\end{prop}
The following lemma guarantees such bounded $n$ and $\Lambda$ exists.
\begin{prop}\label{5.9}
	Let $d,r\in\N$ and $\epsilon>0$ and assume theorem \ref{1.6} hold in dimension $d-1$, then there exists $0<\epsilon'<\epsilon$ and $m,n$ depending only on $d,r,\epsilon$ such that if $(X,B)$ a $\Q$-factorial projective $\epsilon$-lc pair of dimension $d$, $A$ very ample divisors with $A^d\leq r$, $L\geq 0$ $\R$-divisor on $X$ such that $(X,B+tL)$ is $\epsilon'$ lc for some $t<r$, $a(T,X,B+tL)=\epsilon'$ for some $T$ over $X$ where centre of $T$ is a closed point $\{x\}$ on $X$ and $A-B,A-L$ ample, then there exists $\Lambda>0$ such that $n\Lambda$ is integral, $ mA-\Gamma$ is ample, $(X,\Lambda)$ is lc near $\{x\}$ and $T$ is lc place of $(X,\Lambda)$.
\end{prop}
Now we are ready to prove theorem \ref{1.6}.
\begin{proof}[Proof of Theorem \ref{1.6}]
	We use the notation in the statement of \ref{1.6}. It is clear that it suffices to show $lct(X,B,|A|_{\R})\geq t$ for some $t$ bounded away from 0. Firstly by taking a small $\Q$-factorization of $X$, we can assume $X$ is $\Q$-factorial; Indeed we can do this since $X$ is bounded, hence we can choose bounded resolution and hence get a bounded family of small $\Q$-factorisation by \ref{Q_fac} and then replace $A,B$ accordingly. \\
	Step 1 : Let $0\leq L\simR \frac{1}{2} A  $. Let $\epsilon'$ as in Prop \ref{5.9} and let $s$ be the maximum number such that $(X,B+sL) $ is $\epsilon'$-lc. Note it suffices to prove $s$ is bounded away from 0. Choose $T$ divisor over $X$ such that $a(T,X,B+sL)=\epsilon'$. Let $x$ be the generic point of centre of $T$ on $X$. If $x$ is not a closed point, then we are done; Indeed, we can cut by general elements of $|A|$ and apply induction and inversion of adjunction, we get $\exists v$ bounded away from 0, such that $(X,B+vL)$ is lc near $x$, hence we can deduce $s\geq (1-\frac{\epsilon'}{\epsilon} )v $ from $(X,B)$ is $\epsilon$-lc. Hence it remains to consider the case when $x$ is a closed point.\\
	Step 1: Now we can apply \ref{5.9} and by replacing $A,C,L$ by $2mA,2mC,2mL$, we are ready to apply \ref{5.7} to get a bounded $q(d,r,\epsilon)$ such that $\mu_T \phi^* L\leq q$ for some resolution $\phi: W\rightarrow X$ where $T$ is a divisor on $W$. But similarly, we have $\mu_T \phi^* L \geq \frac{\epsilon'-\epsilon}{s}$ from definition of $s$ and $(X,B)$ is $\epsilon$-lc. Hence again we get $s$ is bounded away from 0, hence completing the proof. 	
\end{proof}

We will also note that theorem \ref{BAB} follows easily from \ref{1.4} and \ref{eff_bir}.
\begin{proof}[Proof of BAB]
	We firstly note that the set of $\epsilon$-lc weak Fano variety forms a bounded family; Indeed this follows from \ref{7.13}. All the criteria of \ref{7.13} is satisfied; using theorem \ref{1.4}, theorem \ref{eff_bir}, theorem \ref{com_gen} and \ref{volumn_bnd}. Now when $-(K_X+B)$ is nef and big, we can run MMP on $-K_X$ as in 2.1.6, where we end with $X'$ an $\epsilon$-lc weak Fano, which lies in a bounded family. Hence $X'$ has a klt $n$-complement for some bounded $n$, which implies $X$ has klt $n$-complement by \ref{6.1.3} . Hence such $X$ form a bounded family by \ref{19.1.3}.
\end{proof}

\subsection{Reduction to Toric Case and Proof of \ref{5.7}}
Now we will focus on proving \ref{5.7}, which has some important new ideas. Now we are ready to give the proof of the main theorem 5.7. The main idea here is to show that $T$ can be obtained from $X$ by a bounded number of centred blowup. Then we apply induction on the number of blowup to reduce the problem when $T$ is obtained from a single blowup, and now the question is more or less trivial by elementary methods. 
\\Hence firstly we need a lemma to bound the number of blowup to achieve $T$. This is done via the help of the following lemma. The following lemma is similar to Noether Normalisation lemma that helps to reduce certain problems to projective space.\\
\begin{lemma}\label{5.2}
	Let $(X,\Gamma := \sum_1^d S_i)$ be a projective lg smooth pair of dimension $d$ where $\Gamma$ is reduced. Let $B := \sum b_j B_j \geq 0$ be an $\R$-divisor. Assume $x\in \cap S_i$, supp$B$ contains no stratum of $(X,\Gamma)$ except at $x$, $A$ a very ample divisor such that $A-S_i$ very ample for each $i$, then there is a finite morphism $\pi : X\rightarrow \mathbb{P}^n:= Proj k[t_0,t_1,..,t_d] $ such that $\pi(x)=z := (1:0:0:..:0)$, $\pi(S_i)=H_i := Z(t_i)$, $\pi$ is etale over a neighbourhood of $z$, Supp$B$ contains no points of $\pi^{-1}\{z\}$ except possibly $x$ and deg$\pi=A^d$ and deg$_{H_i} C \leq deg_A B$ where $C := \sum b_j \pi(B_j)$
\end{lemma}
\begin{proof}[Sketch of Proof]
	Taking $D_i\in |A-S_i|$ general and letting $R_i := S_i+D_i\sim A$, we can assume $(X,\sum R_i)$ is log smooth and we may replace $S_i$ by $R_i$ since all condition is satisfied. Furthermore we can choose $R_0 \sim A$ such that $(X,\sum_0^d R_i)$ is still log smooth. $\cap_0^d R_i=\emptyset$. Now $\pi$ is nothing but the map defined by $(f_0,f_1,..,f_d)$, where $div(f_i)=R_i -A$. All of the remaining properties are easy computation using elementary algebraic geometry and definition of etale. 
\end{proof}
Now we can show a lemma that helps to bound the number of blowups.
\begin{lemma}\label{5.5}[Bounds on number of blowups] Let $r,d\in \N,\epsilon>0$ and let $(X,B)$ be $\epsilon$-lc dimension $d$ projective pair with $A$ very ample divisor on $X$ such that $A^d\leq r$. Suppose there is $\Lambda$ such that $(X,\Lambda)$ is log smooth and let $\{x\}$ be a zero dimensional stratum of $(X,\Lambda)$ such that Supp$B$ contain no other stratum of $(X,\Lambda)$. Let $T$ be a lc place of $(X,\Lambda)$ with centre $\{x\}$ and $a(T,X,B)\leq 1$. Suppose further that $deg_A B\leq r$ and $deg_A \Lambda \leq r$. Then there exists $p(d,\epsilon)$ such that $T$ can be obtained by a sequence of centre blowup toroidal with respect to $(X,\Lambda)$ of length at most $p$.
	
\end{lemma}
\begin{proof}
	We essentially apply \ref{5.2} to reduce the problem to toric cases where is the result is clear. Firstly we note it suffices to prove for the case $k=\mathbb{C}$\\
	Step 0: We prove the claim assuming $(X,\Lambda)= (Z := \mathbb{P}^d=Proj(\mathbb{C}[t_0,t_1,..,t_d]), \theta := \sum_1^d H_i)$, where $H_i=Z(t^i)$ and $\{x\}=(1:0:0:..:0)$. By \cite{km} we know we can obtain $T$ by a finite sequence of centred blowup toroidal with respect to $(Z,\theta)$. Let $p$ be the length of the blowup sequence and say we finally achieve $T$ on $\phi: W\rightarrow Z$. Note $W$ is a toric variety. It suffices to show $p$ is bounded from above. Let $E_1,..,E_p$ be exceptional divisors for each blowup and by construction we have $T=E_p$. Hence by \ref{blowup}, it is suffices to bound $\mu_T \phi^* (\theta)$ from above since $\mu_T \phi^* \theta \geq p+1$. Now let $E := \sum_1^{p-1} E_i$ and run toric MMP on $E$ over Z which terminates on $\psi: W'\rightarrow Z$ where the only exception divisor is $T$ since all components of $E$ are contracted by negativity. Let $K_{W'}+B' = \psi^* (K_Z+B)$, where $\mu_T B'>0$ by $a(T,Z,B)\leq 1$ and hence $(W',B')$ is $\epsilon$-lc. Now run a MMP on $-K_{W'}$ ends with a $\epsilon$-lc weak toric fano variety $W''$ ($-K_{W'}$ is big since $-K_W$ is big). Now it is well known that $\epsilon$-lc weak toric fano varieties of dimension $d$ forms a bounded family hence there is a bounded $n$ such that $|-nK_{W''}|$ is base point free by base point freeness (as the Cartier index of $K_{W''}$ is bounded). Hence $K_{W''}$ has klt $n$ complement $\Omega_{W''}$, and by \ref{6.1.3}, $K_{W'}$ has a klt $n$ complement $\Omega_{W'}$ and pushing down to $Z$ gives a klt $n$ complement $\Omega$ for $K_Z$. Hence as $n\Omega$ is integral and $Z$ is smooth, we have $(Z,\Omega)$ is $\frac{1}{n}$-lc and $deg_{H_i} \Omega = -deg_{H_i} K_Z =d+1$ is bounded and $(Z,Supp(\theta+\Omega))$ is bounded. Hence by \ref{4.2}, there is a $u>0$ bounded away from zero such that $(Z,\Omega+u\theta)$ is klt. In particular since $K_{W'}+\Omega_{W'}=\psi^* (K_Z+\Omega)$ and $\Omega_{W'}\geq 0$, we have $\mu_T \phi^* \theta =\mu_T \psi^* \theta\leq \frac{1}{u}$, which shows $p$ is bounded.\\
	Step 1: Now we will reduce the general case to the case above. Firstly since the claim only depends locally near $\{x\}$, we remove any components of $\Lambda$ that doesn't contain $\{x\}$ and assume $\Lambda =\sum_1^d S_i$ and $\cap_i S_i =\{x\}$. Now after replacing $A$ by bounded multiples, we can apply \ref{5.2} and get a $\pi : X\rightarrow Z=\mathbb{P}^d$ etale over $z=(1:0:0:..0)$ with $\pi(x)=z$. Also we can define as in \ref{5.2} $C:=\pi(B)$ and $\theta := \sum_1^d H_i$. Note that $(X,B),(Z,C)$ are analytically isomorphic near $x,z$ and so are $(X,\Lambda),(Z,\theta)$. Hence $(Z,C)$ is $\epsilon$-lc  near $z$ and say $R$ is a lc place of $(Z,\theta)$ corresponding to $T$ with centre $z$. It is easy to check all condition are satisfied by $X,B,\Lambda$ are also satisfied by $Z,C,\theta$ except we only know $(Z,C)$ is $\epsilon$-lc near $z$. \\
	Step 2: Hence we will now modify $C$ to make $(Z,C)$ has good singularities everywhere. We claim there exists $t>0$ such that $(Z,tC+\theta)$ is lc away from $z$. Notice it is clear that we are done given the claim; Indeed, we can let $D=(1-\frac{t}{2})\theta+\frac{t}{2}C$ and we have $(Z,D)$ is $\frac{t}{2}\epsilon$-lc and $-(K_Z+D)$ is ample for $t$ small by looking at degree. Therefore we can construct $\Delta\geq 0$ such that $(Z,\Delta)$ is $\frac{t\epsilon}{2}$-lc and $K_Z+\Delta\simR 0$, and $a(R,Z,\Delta)\leq 1$, which is the case of step 0 by replacing $X,B,\Lambda$ with $Z,\Delta,\theta$. Now we will show the claim. Let $y\in Z\backslash\{z\}$. If $y\not\in Supp\theta$, we can just apply \ref{5.1}. If $y\in Supp\theta$, let $G$ be the stratum of $(Z,\theta)$ that contains $y$ of the smallest dimension, note $\dim G\geq 1$ since $y\neq z$. Then by adjunction, we get $K_G+tC|_G = (K_Z+\theta+tC)|_G$ near $y$. Therefore we are done again by applying inversion of adjunction and induction. This proves the claim and finishes the proof of the proposition.
\end{proof}
Now we will use \ref{5.5} to show \ref{5.7}.
\begin{proof}[Proof of \ref{5.7}]
	Step 0: We firstly reduced to problem to something that we can almost apply \ref{5.5}. First by induction and cutting by hyperplane sections of general members of $|A|$, we can assume $\{x\}$ is a closed point. Also since $n\Lambda$ is integral, we have $(X,supp(A+\Lambda))$ is bounded. Let $W\rightarrow X$ be a log resolution of $(X,\Lambda)$ such that there exists $A_W$ very ample on $W$ and $\theta_W := Supp(\Lambda)^\sim +exc/X$ such that $(W,A_W+\theta_W)$ form a bounded family. Since $(X,\Lambda)$ is $lc$ near $x$, we have $a(T,W,\theta_W)=0$ hence $T$ is a lc place of $(W,\theta_W)$ and let $G$ be its centre on $W$. Let $K_W+B_W = (K_X+B)|_W$. We have coefficients of $B_W$ are bounded from below since $A-B$ is ample and $X$ is bounded family (note here we are using our specific choice of bounded $W$). Hence there exists $t>0$ sufficiently small depending only on $d,r,\epsilon$, such that $\Delta_W := tB_W+(1-t)\theta_W \geq 0$ and $a(T,W,\Delta_W)=ta(T,W,B_W)\leq t\leq 1$. Let $L_W := L|_W$, we can now replace $(X,B,\Lambda,A,x,L)$ by $(W,\Delta_W,\theta_W,A_W,G,L_W)$. Hence from now on we can assume $(X,\Lambda)$ is log smooth and $\Lambda$ is reduced and by applying above again assume $\{x\}$ is a closed point. Note that we are ready to apply \ref{5.5} but need supp$B$ not contain any stratum of $(X,\Lambda)$ other than $x$.
	\\\\
	Step 1: Here we reduce to the case when Supp$B$ only contain a bounded number of zero dimensional stratum of $(X,\Lambda)$. By induction and cutting by $H\in |A|$ general hyperplane and apply inversion of adjunction,  we can assume there is a $t(d,r,\epsilon)>0$ such that $(X,B+tB)$ is $\epsilon/2$-lc except at finitely and bounded many closed points. Let $\psi: V\rightarrow X$ be a log resolution of $(X,B)$ where $T$ is a divisor on $V$. Let $\Gamma_V =(1+t)B^\sim+(1-\frac{\epsilon}{4})\sum_i E_i +(1-a)T$ where $E_i$ are all the exceptional divisors other than $T$ and $a=a(T,X,B)$. Similar to proof of \ref{5.9}, we have $K_V+\Gamma_V = \psi^* (K_X+B) +tB^\sim+F = \psi^* (K_X+B+tB)+G$, where $F:= \sum_i (a(E_i,X,B)-\epsilon/4)E_i\geq 0$ and $T\notin Supp F$ and $G:= \sum_i (a(E_i,X,(1+t)B)-\epsilon/4 )E_i +(a'-a)T$ where $a':= a(T,X,(1+t)B)$. Notice that if $dim \psi(E_i)>0$, then $\mu_{E_i} G\geq 0$ by construction of $t$ and $F,G$ both exceptional over $X$. Now we can run MMP on $K_V+\Gamma_V$ over $X$ and we end up with minimal model $Y$ since $K_V+\Gamma_V\simR tB^\sim+ F/X \geq 0$. Notice that this is also a MMP on $G$ hence by negativity, we see that this MMP contracts all terms of $G$ with positive coefficients, hence $\pi : Y\rightarrow X$ is an isomorphism over $X\backslash\{ \text{finitely many closed points} \}$. Also $T$ is not contracted since $T\notin Supp(tB^\sim+F)$. Let $A_Y := A|_Y$, by boundedness of extremal rays and base point freeness, we have $K_Y+\Gamma_Y+3dA_Y$ is nef and big and semi ample globally. Pick $0\leq D_Y\simR \frac{1}{t}(K_Y+\Gamma_Y+3dA_Y)$ with coefficients of $D_Y\leq 1-\epsilon$ and $D_Y$ general. Also let $D:=\pi_* D_Y$. We see that Supp$D$ contain no positive dimensional stratum of $(X,\Lambda)$ since $\pi$ is isomorphism over X except over finitely closed points. It can be easily verified that $D \simR B+\frac{1}{t}(K_X+B+3dA)$ and $a(T,X,D)=a(T,X,B)\leq 1$. Also there is a bounded $m\in\N$ such that $mA-D$ is ample since $D\simR B+\frac{1}{t}(K_X+B+3dA)$. Hence we can replace $B$ and $D$ and assume Supp$B$ contain no positive dimensional stratum of $(X,\Lambda)$ except at possibly bounded finitely many closed points.
	\\\\
	Step 2: Now we show that we can always reduce to the case when Supp$B$ contain no stratum of $(X,\Lambda)$ other than $x$. Say $y\neq x$ is such a stratum. Let $X'\rightarrow X$ be the blowup at $y$ and $E'$ is the exceptional divisor and let $K_{X'}+B'$, $K_{X'}+\Lambda',L'$ be the pullback to $X'$. Then we have $\mu_{E'} B' \geq -d+1 $ by the blowup formula and $\mu_{E'} \Lambda'=1$, hence by choosing $\beta := \frac{1}{2d}$ say, we have $B'' := \beta B'+(1-\beta)\Lambda' \geq 0$ and $(X,',B'')$ is $\beta\epsilon$-lc. Now replace $X,B,\Lambda,L,\epsilon$ by $X',B'',\Lambda',L',\beta\epsilon$, we can remove one of zero dimensional stratum of $(X,\Lambda)$, which is contained in Supp$B$. Now repeating this process a bounded number of times and applying step 1, we can assume Supp$B$ contain no stratum of $(X,\Lambda)$ except $\{x\}$. This puts allows us to apply step 3 to finish the proof. \\\\
	Step 3: Firstly we settle the case when Supp$B$ contains no stratum of $(X,\Lambda)$ except possibly at $x$. Here we can apply \ref{5.5}, and get there exists $p(d,r,\epsilon)$ such that $v: V =X_l\rightarrow ...\rightarrow X_1\rightarrow X_0=X$ a sequence of centred blowup such that $l\leq p$ and $T$ is a divisor on $X_l$. In particular $X_1\rightarrow X$ is a blowup at $x$. We can define as in step 2, $B_1' := \beta B_1+(1-\beta)\Lambda_1\geq 0$ and $(X_1,B_1')$ is $\beta\epsilon$-lc and $l$ is reduced by 1. If centre of $T$ not a closed point, we apply step 0 and we are done. If centre of $T$ is still a closed point we apply Step 2 and Step 3 and apply induction on length of blowup. Since the length of blowup is bounded above by $p$, we know for at most repeating this whole process $p$ times, we will get to $x$ is not a closed point, hence we are done by Step 0. Hence eventually, we can find a $q(d,r,\epsilon,p)>0$ such that $\mu_T v^*L \leq q$. This completes the proof.
	
\end{proof}

\newpage
\subsection{Construction of $\Lambda$ from complements }
\begin{proof}[Proof of \ref{5.9}]
	We first follow the same idea as in proof of \ref{5.7}. We construct some birational model $Y$ where we can apply the theorem of complements as in \ref{1.7}.\\
	Step 0: Let $\epsilon'$ be given as in ACC of log canonical threshold in the sense that if $(Y,(1-\epsilon')S)$ is $\Q$-factorial projective klt of dimension d, with $S$ reduced then $(Y,S)$ is lc. Applying same arguments as in proof of \ref{5.7}, we choose a $v>0$ depending only on $d,r,\epsilon$ such that $(X,(1+v)(B+tL)$ is $\frac{\epsilon}{2}$-lc except possibly at finitely many points. Let $W\xrightarrow{\phi} X$ be log resolution of $(X,B+L)$ such that $T$ is a divisor on $W$. Define $\Gamma_W := (1+v)(B^\sim +t L^\sim)+\sum (1-\frac{\epsilon'}{4}) E_i+(1-\epsilon') T$, where $E_i$ are all expception divisor over $X$ on $W$. It is easy to see that $K_W+\Gamma_W \equiv \phi^*(K_X+B+tL)+v(B^\sim+tL^\sim)+F \equiv \phi^*(K_X+(1+v)(B+tL)) +G $, where $F\geq 0$ exceptional over $X$, $T$ not a component of $F$, $G$ exceptional over $X$, and all $E_i$ components of $G$, which has positive dimension centre on $X$ have positive coefficients. Now run MMP on $K_W+\Gamma_W$, we end with a model $Y'\rightarrow X$ where all such $E_i$ with positive dimension centre are contracted by negativity, hence $Y'\rightarrow X$ is isomorphism over the complements of finitely many points. Also note since $T$ not a component of $F$, $T$ is still a divisor on $Y'$.Let $A_{Y'}:= A|_{Y'}$\\
	Step 1: Let $\psi: Y\rightarrow X$ be the extremal extraction of $T$ and let $A_Y :=A|_Y$. Notice that by construction $(Y,(1-\epsilon
	')T)$ is klt hence we have $(Y,T)$ lc. In order to apply \ref{1.7}, we will now show $\exists \;l$ bounded such that $lA_Y-(K_Y+T)$ is ample. Let $K_Y+\Gamma_Y$ and $A_Y$ be pushdown of $K_{Y'}+\Gamma_{Y'} $ and $A_{Y'} $. We claim that $K_Y+\Gamma_Y+3dA_Y$ is in fact ample; Indeed, if $C$ is a curve on $Y$ that is contracted over $X$, then $(K_Y+\Gamma_Y+3dA_Y)\cdot C =v(B^\sim +tL^\sim)\cdot C\geq vtL^\sim \cdot C>0$ (since $F_Y=0$, $\mu_T \psi^* L>0$ by construction, $0=\psi^* L \cdot C = L^\sim \cdot C + \mu_T \psi^* L (T\cdot C)$ and $T\cdot C<0$ since the contraction is extremal). Now if $C$ is not contracted over $X$, then $A_Y\cdot C>0$ and $(K_Y+\Gamma_Y+2dA_Y)\cdot C = (K_{Y'}+\Gamma_{Y'}+2dA_{Y'})\cdot C \geq  0$ since $Y'\dashrightarrow Y$ is isomorphism over $\eta_C$ and $K_{Y'}+\Gamma_{Y'}+2dA_{Y'} $ is globally semi ample by base point free theorem and bounds on length of extremal rays. \\
	Observe we have from construction $K_Y+\Gamma_Y = \psi^* (K_X+B+tL)+v(B^\sim+tL^\sim)$ and $K_Y+B^\sim+tL^\sim+(1-\epsilon')T = \psi^*(K_X+B+tL)$. Also noting that $A-B,A-L$ both ample implies that there is a bounded $l'$ such that $l' A_Y-(K_Y+(1-\epsilon')T)$ is ample since $$l'A_Y-(K_Y+(1-\epsilon')T) = \psi^*[(l'-\frac{3d}{v})A-(1+\frac{1}{v})(K_X+B+tL)]+\frac{1}{v}(K_Y+\Gamma_Y+3dA_Y)$$ Hence there is a $l$ such that $3lA_Y -(K_Y+T)$ is ample since $$3lA_Y-(K_Y+T) = (lA_Y-(K_Y+(1-\epsilon')T)) +(lA_Y-\alpha\psi^*(B+tL))+(lA_Y+\alpha(B^\sim+tL^\sim)) $$ where $\alpha$ is such that $\mu_T \alpha \psi^*(B+tL) =\epsilon'$ (note $\alpha$ is bounded from above). The first summand is ample by above, the second is clearly nef and the last is also ample since $lA_Y+\alpha(B^\sim+tL^\sim) = [(l-\frac{3d}{v})A_Y-\frac{1}{v}(\psi^*(K_X+B+tL))]+\frac{\alpha}{v}(K_Y+\Gamma_Y+3dA_Y)$.  \\
	Step 2: We can now apply \ref{1.7} and get a $0\leq P_Y\sim (n+1)lA_Y-n(K_Y+T)$ such that $(Y,T+\frac{1}{n}P_Y)$ is lc near $T$. Let $\Lambda $ by pushdown of $\Lambda_Y := T+\frac{1}{n}P_Y$. Finally it can be easily verified that $\Lambda$ satisfies all the claim we want.
\end{proof}
\newpage

%%%%%%%%%%%%%%%%%%%%%%%
%%%%%%%%%%%%%%%%%%%%%%%%%%%

%%%%%%%%%%%%%%%%%%%%%%%
%%%%%%%%%%%%%%%%%%%%%%%
%%%%%%%%%%%%%%%%%%%%%%%

\bibliography{Fano_1}
\bibliographystyle{plain}
\end{document}